\documentclass[a4paper,twoside]{article}
\usepackage{pdfsync}
\usepackage{amsmath,amsthm,amsfonts,latexsym,amscd,amssymb,enumerate}

\usepackage[pdftitle={The surgery exact sequence, K-theory and the
  signature operator},pdftex]{hyperref}

\swapnumbers
\theoremstyle{plain}
\newtheorem{theorem}{Theorem}[section]
\newtheorem{lemma}[theorem]{Lemma}
\newtheorem{corollary}[theorem]{Corollary}
\newtheorem{proposition}[theorem]{Proposition}

\theoremstyle{definition}
\newtheorem{definition}[theorem]{Definition}
\newtheorem{example}[theorem]{Example}
\newtheorem{notation}[theorem]{Notation}

\newtheorem{set-up}[theorem]{Geometric set-up}
\newtheorem{remark}[theorem]{Remark}

\newcommand{\relC}[2]{C^*(#2\subset {#1})} % here the first entry is the big space,
\newcommand{\relD}[2]{D^*(#2\subset{#1})} % here the first entry is the space, the
\DeclareMathOperator{\coarseind}{Ind}

\newcommand{\reals}{\mathbb{R}}
\newcommand{\complexs}{\mathbb{C}}
\newcommand{\naturals}{\mathbb{N}}
\newcommand{\integers}{\mathbb{Z}}

\newcommand{\KK}{\mathbb{K}}

\DeclareMathOperator{\id}{id}
\newcommand{\boundary}[1]{\partial#1}

\newcommand{\abs}[1]{\left\lvert#1\right\rvert} %absolute value
\newcommand{\norm}[1]{\left\lVert#1\right\rVert}

\newcommand{\tensor}{\otimes}
\newcommand{\into}{\hookrightarrow}

\newcommand{\iso}{\cong}

\newcommand{\disjointunion}{\sqcup}

\DeclareMathOperator{\supp}{supp}   %support
\DeclareMathOperator{\im}{im}      %image
\DeclareMathOperator{\spin}{spin}

\DeclareMathOperator{\Pos}{Pos} 
\DeclareMathOperator{\spec}{Spec}

\DeclareMathOperator{\Ind}{Ind}

\newcommand{\forget}[1]{}

\def  \nuint {\raise10pt\hbox{$\nu$}\kern-6pt\int}

\def \L{\mathcal L}

\newcommand\C{\mathcal C}

\def \L {{\cal L}}
\def \Sp {{\cal S}}

\def\Id{{\rm Id}}

\newcommand\D{\mathcal D}
\newcommand\Di{D\kern-6pt/}
\newcommand\cDi{{\mathcal D}\kern-6pt/}
\newcommand\spi{S\kern-6pt/}
\newcommand \cspi{\Sp\kern-6pt/}

\newcommand\CC{\mathbb C}

\def \cal {\mathcal}
\def \C {{\cal C}}

\def \BB {\mathbb B}

\newcommand\RR{\mathbb R}
\newcommand\ZZ{\mathbb Z}

\newcommand\pa{\partial}

 \usepackage{color}
\definecolor{darkgreen}{cmyk}{1,0,1,.2}
\definecolor{m}{rgb}{1,0.1,1}
{\catcode`@=11\global\let\c@equation=\c@theorem}

\allowdisplaybreaks[2]

\begin{document}
\pagestyle{myheadings}
\markboth{Paolo Piazza and Thomas Schick}{Surgery sequence, K-theory  and the signature  operator}

%\date{Last compiled \today; last edited  \heuteIst or later}

\title{The surgery exact sequence, K-theory  and the signature  operator}

\author{Paolo Piazza and Thomas Schick}
\maketitle

\begin{abstract}
The {main result} of this paper is a new and direct proof of the natural
transformation from the surgery exact sequence in topology to the analytic
K-theory sequence of Higson and Roe.

Our approach makes crucial use of analytic properties and new index
theorems for the signature operator
on Galois coverings  with boundary. These are of independent interest and 
form the second main theme
of the paper. The main technical novelty is the use of large scale index
theory for 
Dirac type operators that are perturbed by lower order operators.
\end{abstract}

\tableofcontents

\section{Introduction}
Let $V$ be a smooth, closed, oriented $n$-dimensional manifold.
We consider $\Gamma:=\pi_1 (V)$ and the universal cover $\widetilde{V}\to V$.
Finally, we let $\tilde{u}\colon \widetilde{V}\to E\Gamma$ be 
a $\Gamma$-equivariant
map covering a classifying map $u\colon V\to B\Gamma$ for $\widetilde{V}\to V$.
Associated to this data there are two important exact sequences.

\medskip
The first one, due to Browder, Novikov, Sullivan and Wall, is the surgery exact sequence
in topology \cite{MR2000a:57089}, \cite{ranicki-ags}, \cite{lueck-survey}:
\begin{equation}\label{wall-sequence}
\dots\rightarrow L_{n+1}(\ZZ\Gamma)  \dashrightarrow \mathcal{S}(V) \rightarrow
 \mathcal{N}(V) \rightarrow  L_{n}(\ZZ\Gamma)\,.
\end{equation}
The central object of interest in this sequence is the {\em structure set} $\mathcal{S}(V)$; elements in  the set $\mathcal{S}(V)$ are given by homotopy manifold structures on $V$, i.e. orientation preserving homotopy equivalences 
$f\colon M\rightarrow V$, with $M$ a smooth oriented closed manifold, considered
up to $h$-cobordism.  $\mathcal{N}(V)$ is the set of degree one normal maps 
$f:M\to V$ considered up to normal bordism. Finally, the abelian 
groups $L_* (\ZZ\Gamma)$, the $L$-groups of the integral group ring
$\ZZ\Gamma$, are defined
algebraically but have a geometric realization as cobordism groups 
of manifolds with boundary with additional structure on the boundary.

The surgery exact sequence \eqref{wall-sequence} plays a fundamental role in the classification
of high-dimensional smooth compact manifolds.

\medskip
The second exact sequence associated to $\widetilde{V}\to V$ is purely analytic
and  is due to Higson and Roe. Consider  the $C^*$-algebra
$C^*(\widetilde{V})^\Gamma$ called the Roe algebra and obtained as the closure of the $\Gamma$-equivariant locally compact finite propagation operators; this is an ideal in
$D^*(\widetilde{V})^\Gamma$, the $C^*$-algebra obtained as the closure 
 of the $\Gamma$-equivariant pseudolocal finite propagation
operators. There is a short exact sequence of $C^*$-algebras 
$$0\rightarrow  C^*(\widetilde{V})^\Gamma\rightarrow 
D^*(\widetilde{V})^\Gamma\rightarrow 
D^*(\widetilde{V})^\Gamma/C^*(\widetilde{V})^\Gamma\rightarrow 0$$ and thus a 6-terms long exact sequence in K-theory
$$\cdots\rightarrow K_{n+1}  (C^*(\widetilde{V})^\Gamma)
\rightarrow K_{n+1}  (D^*(\widetilde{V})^\Gamma)
\rightarrow K_{n+1}  (D^*(\widetilde{V})^\Gamma/C^*(\widetilde{V})^\Gamma)
\xrightarrow{\pa} K_n (C^*(\widetilde{V})^\Gamma)\rightarrow \cdots$$
There are canonical isomorphisms,
$$K_{*+1}  (D^*(\widetilde{V})^\Gamma/C^*(\widetilde{V})^\Gamma)= K_{*} (V)
\quad\text{and}\quad  K_{*}  (C^*(\widetilde{V})^\Gamma)=K_* (C^*_r \Gamma)$$
with $C^*_r \Gamma$ denoting, as usual, the reduced $C^*$-algebra of the
group $\Gamma$.
Thus we can rewrite the long exact sequence in K-theory as
\begin{equation}\label{HR}
\cdots\rightarrow K_{n+1}  (C^*_r \Gamma)
\rightarrow K_{n+1}  (D^*(\widetilde{V})^\Gamma)
\rightarrow K_{n}  (V)
\rightarrow K_n (C^*_r \Gamma)\rightarrow \cdots
\end{equation}
This is the {\em analytic surgery sequence} of Higson and Roe. 

% Using the classifying map $u\colon V\to B\Gamma$ 
% and its lift $\tilde{u}\colon \widetilde{V}\to E\Gamma$ we can 
% map \eqref{HR} to the {\em universal} analytic surgery sequence
% \begin{equation}\label{HR-universal}
% \cdots\rightarrow K_{n+1}  (C^*_r \Gamma)
% \rightarrow K_{n+1}  (D^*_\Gamma)
% \rightarrow K_{n}  (B\Gamma)
% \rightarrow K_n (C^*_r \Gamma)\rightarrow \cdots
% \end{equation}
% with $ K_*(D^*_\Gamma):= \dirlim_{X\subset E\Gamma\text{ $\Gamma$-compact}}
% K_*(D^*(X)^\Gamma)$.
By \cite{roe-bc_coarse}, the  map 
$K_* (V)\rightarrow K_* (C^*_r \Gamma)$  is precisely  
the (Baum-Connes) assembly map. This connects the Higson-Roe surgery sequence
to fundamental questions such as the Strong Novikov Conjecture or the
Baum-Connes conjecture.

In a series of papers \cite{higson-roeI,higson-roeII,higson-roeIII} 
Higson and Roe constructed the following remarkable commutative diagram:
       \begin{equation}\label{HRdiagram}
  \begin{CD}
  %@>>>
   L_{n+1}(\ZZ\Gamma) &\dashrightarrow& \mathcal{S}(V) @>>> \mathcal{N}(V) @>>> L_{n}(\ZZ\Gamma)\\
   %&& 
   @VV{\gamma}V  @VV{\alpha}V @VV{\beta}V 
   @VV{\gamma}V \\
   %@>>> 
   K_{n+1} ( C^*_r\Gamma)[\frac{1}{2}]  @>>>  K_{n+1}(D^* (\widetilde{V})^\Gamma )[\frac{1}{2}]
   @>>>K_{n} (V)[\frac{1}{2}] @>>>  K_{n} ( C^*_r\Gamma)[\frac{1}{2}] \\
    \end{CD}
    \end{equation}
{\em where $A[\frac{1}{2}]$ is a shorthand for
$A\tensor_\integers\integers[\frac{1}{2}]$ if $A$ is any abelian group. }

% Using the 
% classifying map $u$ and its lift $\tilde{u}$ there is a similar diagram involving
% the universal analytic surgery sequence:

%  \begin{equation}\label{HRdiagram-universal}
%   \begin{CD}
%   %@>>>
%    L_{n+1}(\ZZ\Gamma) @>>> \mathcal{S}(V) @>>> \mathcal{N}(V) @>>> L_{n}(\ZZ\Gamma)\\
%    %&& 
%    @VV{\gamma_\Gamma}V  @VV{\alpha_\Gamma}V @VV{\beta_\Gamma}V 
%    @VV{\gamma_\Gamma}V \\
%    %@>>> 
%    K_{n+1} ( C^*_r\Gamma)[\frac{1}{2}]  @>>>  K_{n+1}(D^*_\Gamma )[\frac{1}{2}]
%    @>>>K_{n} (B\Gamma)[\frac{1}{2}] @>>>  K_{n} ( C^*_r\Gamma)[\frac{1}{2}] \\
%     \end{CD}
%     \end{equation}
It is important to mention that both $\alpha$ and 
$\gamma$ in \eqref{HR} are constructed using fine properties of Poincar\'e
spaces that are
\emph{not} smooth manifolds. The main ingredient is the use of homotopy
equivalences to glue manifolds with boundary along their boundaries. The
resulting objects are \emph{not} manifolds, but still Poincar\'e
complexes. These have well defined (higher) signatures which then feature in
the construction.

% In order to clarify this point,
% let us describe in some detail the homomorphism $\gamma
% \colon L_{n+1}(\ZZ\Gamma)\to  K_{n+1} ( C^*_r\Gamma)[\frac{1}{2}] $ 
% with $n$ odd: a
% cycle for $L_{n+1}(\ZZ\Gamma)$ is given by a quadruple 
% $(W,F,X\times [0,1],u\colon
% X\to B\Gamma)$ 
% with $W$ a $n+1$-dimensional smooth cobordism between two smooth orientable manifolds $\pa_1 W$ and $\pa_2 W$, $X$ a smooth orientable manifold,
% $F\colon (W, \pa W) \to (X\times [0,1], \pa(X\times [0,1]))$ a degree one map of pairs, $f_1:=F\,|_{\pa_1 W}$ and $f_2:=F\,|_{\pa_2 W}$ oriented homotopy equivalences
% and $u\colon X\to B\Gamma$ a classifying map. We glue $W$ to $X\times [0,1]$
% through the homotopy equivalences $f_1$ and $f_2$ and we obtain a closed
% Poincar\'e space $Z$ endowed with a classifying map $u_Z$ into $B\Gamma$. Such
% a space, albeit \emph{not} a smooth manifold,
% has a well defined Mishchenko signature $\sigma_{C^*_r\Gamma} (Z,u_Z)
% \in K_{n+1} (C^*_r\Gamma)$ and the map $\gamma$ in the Higson-Roe
% commutative diagram is precisely 
% $$\gamma \,[W,F,X\times [0,1],u\colon
% X\to B\Gamma] := \sigma_{C^*_r\Gamma} (Z,u_Z)
% \;\in\; K_{n+1} (C^*_r\Gamma)\,.$$
% The map $\alpha$ is even more complicated but it also involves Poincar\'e
% spaces that are not smooth manifolds.

% \medskip
% \noindent
{\em One main  goal of this article 
is to give an alternative and direct route to the transformation from the
smooth surgery exact 
sequence in topology to the analytic surgery sequence of Higson-Roe.
The main point of our construction is that we use index theoretic
constructions, applied to the signature operator, throughout.}

%\medskip
Our approach follows the one presented in our recent paper
 \cite{PS-Stolz}, where we showed how to
map the exact sequence of Stolz for positive scalar curvature metrics
to the Higson-Roe exact sequence. Throughout, we will follow the notations
of our  paper \cite{PS-Stolz}, the current paper should be considered as a
companion to the latter one, with the signature operator replacing the Dirac
operator in a fundamental 
way.

\smallskip
Our main result, Theorem \ref{theo:main}, is that there are natural (index
theoretic) maps $\Ind_\Gamma, \rho,\beta$ 
    
    \begin{equation}\label{eq:our_diag}
  \begin{CD}
  %@>>>
   L_{n+1}(\ZZ\Gamma) &\dashrightarrow & \mathcal{S}(V) @>>> \mathcal{N}(V) @>>> L_{n}(\ZZ\Gamma)\\
   %&& 
   @VV{\Ind_\Gamma}V  @VV{\rho}V @VV{\beta}V 
   @VV{\Ind_\Gamma}V \\
   %@>>> 
   K_{n+1} ( C^*_r\Gamma)[\frac{1}{2}]  @>>>  K_{n+1}(D^*(\tilde V)^\Gamma
   )[\frac{1}{2}] 
   @>>>K_{n} (V)[\frac{1}{2}] @>>>  K_{n} ( C^*_r\Gamma)[\frac{1}{2}] \\
    \end{CD}
    \end{equation}
%
%    
%    
%    \begin{CD}
%   @>>> \Omega^{\spin}_{n+1} (B\Gamma) @>>>   R^{\spin}_{n+1}(B\Gamma) @>>>
%    \Pos^{\spin}_n (B\Gamma) @>>> \Omega^{\spin}_n (B\Gamma) @>>>\\ 
%    && @VV{\beta}V @VV{\Ind_\Gamma}V  @VV{\rho_\Gamma}V @VV{\beta}V\\
%     @>>>  K_{n+1} (B\Gamma) @>>> K_{n+1} ( C^*_r\Gamma) @>>>
%     K_{n+1}((D^*_\Gamma))  @>>>  K_{n} ( B\Gamma) @>>>\\
%    \end{CD}
%\end{equation*}
making the diagram commutative.

The main technical novelty, compared to the companion paper \cite{PS-Stolz},
is 
that we have to use an orientation preserving homotopy $f\colon M\to V$ to
perturb the signature operator (on the disjoint union of $M$ and $V$) to an invertible
operator. {We follow an explicit recipe for such a perturbation $C_f$ initiated
by 
Hilsum-Skandalis in their fundamental work \cite{HiSka}, and then modified by Piazza and Schick in \cite{PS1} and further modified by Wahl in \cite{Wahl_higher_rho}.
We will use the latter version  of this perturbation and for this reason
we name it Wahl's
perturbation. One advantage of choosing  Wahl's perturbation is that
it makes it possible to extend most of the results
of Hilsum and Skandalis  from closed manifolds to manifolds with boundary. See 
\cite[Theorem 8.4]{Wahl_higher_rho}.}

Notice that the perturbed
operator looses several of the appealing properties of a Dirac type operator
(like unit propagation of the associated wave operator). 
{Therefore, the second main theme of this work is the proof of a new coarse
 index theorem for the signature operator
on Galois coverings  with boundary. This is of independent interest.} The main technical novelty is the use of large scale index
theory for Dirac type operators that are perturbed by lower order operators.

\smallskip
More explicitly, in \eqref{eq:our_diag} the map $\rho$ is defined applying the
following idea: if $M\xrightarrow{f} V$ is a homotopy equivalence 
then $\rho [ M\xrightarrow{f} V]$ is defined in terms of the
projection onto the positive part of the spectrum of the self-adjoint
invertible operator $D+C_f$, with $D$
the signature operator on the Galois covering defined by 
$(f\circ u) \disjointunion (-u)\colon  M\disjointunion (-V)\to B\Gamma$ and
$C_f$  {is Wahl's perturbation} defined by $f$.

\smallskip
A geometrically given cycle for $L_n(\integers\Gamma)$ consists in particular
of a manifold with boundary (made up of two components), and the extra datum
of a homotopy equivalence of the boundaries of the pieces.
%Hilsum-Skandalis 
{Wahl's}
perturbation for this homotopy equivalence can then be used
to perturb the signature operator to be invertible at the boundary. 
%
% The homomorphism $\Ind$ is a {\it generalized} Atiyah-Patodi-Singer index class
% on the manifold with boundary $W\disjointunion X\times [0,1]$. Indeed, the
% restriction to the boundary of the 
% map $F\colon W\to X\times [0,1]$ entering into the definition of a class
% $\,[W,F,X\times [0,1],u\colon
% X\to B\Gamma]$ in $L_{n+1}(\ZZ\Gamma)$, allows to construct a Hilsum-Skandalis
% perturbation for the (Mishchenko-Fomenko) signature operator on the boundary of  $W\cup X\times [0,1]$;
% this perturbation makes the boundary operator invertible 
% and 
This allows for the definition of a generalized Atiyah-Patodi-Singer index class
{$\Ind(D,f)$} in $K_{n+1} (C^*_r \Gamma)$, which in the end defines the map $\Ind_\Gamma$.% ; this class is, by definition,  the image of $
% [W,F,X\times [0,1],u\colon
% X\to B\Gamma]$ through $\Ind_\Gamma$.

\smallskip
Finally, $\beta$ is defined 
as in Higson and Roe: if $f\colon M\rightarrow V$ defines a class
in $\mathcal{N} (V)$ then its image through $\beta$
is obtained as $f_* ([D_M])- [D_V]\in K_n (V)$, with $[D_M]$
and $[D_V]$ the fundamental classes associated to the signature operators on the smooth compact manifolds
$M$ and $V$.

%is equal to $u_* (f_* ([D_M])- [D_V])\in K_n (B\Gamma)$, with $D_M$
%and $D_V$ the signature operators on the smooth compact manifolds
%$M$ and $V$.

\smallskip
We shall first treat the case in which $V$ is odd
dimensional and only at the end indicate how the results of this paper 
and of \cite{PS-Stolz} can also be extended to the even dimensional case.

\smallskip
It should be added that the map $\Ind_\Gamma$
of our diagram, the one out of $L_*(\ZZ\Gamma)$, has 
already been constructed by Charlotte Wahl in \cite{Wahl_higher_rho}. We 
will make use of her important results for parts of our program. 
{The main novelty in the approach presented in this paper and in
\cite{Wahl_higher_rho}, compared 
to \cite{higson-roeI,higson-roeII,higson-roeIII},  is therefore the definition
of the map $\Ind_\Gamma$, the definition 
of $\rho$ and, crucially, the proof of well-definedness and of
commutativity of the squares. }

All this we prove by establishing and then employing
 a delocalized Atiyah-Patodi-Singer index theorem for
\emph{perturbed} Dirac operators:

\begin{theorem}[Theorem \ref{theo:k-theory-deloc}]
Let $W$ be an oriented manifold with free cocompact orientation preserving
action and with boundary $M_1\disjointunion M_2$. Let $f\colon M_1\to M_2$ be
an orientation preserving $\Gamma$-equivariant homotopy equivalence. Then
\begin{equation*}
\iota_* ( \Ind(D,f) )= j_*(\rho( D_{\pa} + C_{f})) \quad\text{in}\quad K_{0}
(D^*(\widetilde{W})^\Gamma).
\end{equation*}

Here, $D$ is the signature operator on $W$, {$\Ind(D,f)$ is the generalized
APS index class associated to $D$ and to the homotopy equivalence $f$,
 $j\colon D^*(\boundary
\widetilde{W})^\Gamma\to 
D^*(\widetilde{W})^\Gamma$ is the homomorphism induced by the inclusion 
$\boundary \widetilde{W}\to \widetilde{W}$} and $\iota\colon
C^*(\widetilde{W})^\Gamma\to D^*(\widetilde{W})^\Gamma$ the inclusion.

\end{theorem}

{It should be added that this result 
generalizes to general Dirac type operators with more abstract boundary
perturbations making the boundary operator invertible. See
Theorem \ref{theo:k-theory-deloc} for a precise statement. }

The main novelty in Theorem \ref{theo:main}, {i.e.~the construction and the commutativity
of the  diagram
\eqref{eq:our_diag},
 compared to the analogous result in \cite{PS-Stolz},} is the treatment of
the technicalities which arise when dealing with perturbed Dirac type
operators.%  Notice that our approach can be linked directly to
% existing higher Atiyah-Patodi-Singer index formulae, bringing
% into the whole picture 
% primary and secondary numeric invariants of the signature operator
% on Galois coverings.

\begin{remark}
By using the signature operator of Hilsum and Teleman
it is possible to extend our results to Lipschitz manifolds and thus,
by Sullivan's theorem, to the surgery exact sequence in the category TOP of
topological manifolds. {We refer 
the reader to Zenobi's
\cite{zenobi} where stability results for the topological structure set under
products are also discussed.}
\end{remark}

\begin{remark}
{ Because of their fundamental result,
Higson and Roe name the group $K_{0} (D^*_\Gamma)$
 the {\it analytic structure group}; they denote it $\mathcal{S}_1(\Gamma)$.
Similarly, the analytic structure group $\mathcal{S}_0 (\Gamma)$ is, by definition,
the group $K_{1} (D^*_\Gamma)$.
In \cite{DGof1}, Deeley and Goffeng have introduced a  geometric 
 structure group $\mathcal{S}^{{\rm geo}}_1(\Gamma)$, with cycles
 defined \`a la Baum-Douglas.
 Using the results of the present paper, Deeley and Goffeng have proved that their geometric  
definition is isomorphic to the analytic one given by Higson and Roe.}
\end{remark}

%%list of sections begins
%\section{Index and rho classes defined by perturbations}
%\label{sec:indexdef}
%\subsection{The index homomorphism in L-theory}\label{sub:L-map}
%\subsection{Rho classes}\label{sub:coarse-index}
%\subsubsection{Perturbations on the covering}
%\subsubsection{The  rho-class associated to a perturbation}
%\subsubsection{Fundamental examples of rho classes}
%\subsection{Index classes}\label{subsub:coarse-classes}
%\subsubsection{Coarse index classes on manifolds with cylindrical ends}\label{subsub:cylindrical-classes}
%\subsubsection{Compatibility of index classes}\label{subsub:compatibility}
%\section{Delocalized APS-index theorem for perturbed operators}
%\section{Mapping the surgery exact sequence to K-Theory}
%\subsection{The structure set and the map $\rho_{\Gamma}$}
%\subsection{The group $\protect\mathcal{N} (V)$ and the map $\beta_\Gamma : \protect\mathcal{N} (V) \to K_n (\protect\Gamma)$}
%\subsection{Mapping the surgery exact sequence to the Higson-Roe
%   sequence}
% \section{Proof of the delocalized APS index theorem}\label{sec:proof}
%    \subsection{Reduction to the cylinder}
%\subsection{Proof of the cylinder delocalized index theorem for perturbed operators}
% \subsubsection{Proof of Proposition \ref{prop:v-covers}: the map $V$ for the
%   perturbed operator covers the identity in the  $D^*$-sense}\label{sec:V_covers}
%\subsubsection{Proof of Propositions  \ref{prop:deformed-in-d}: the operator $\frac{\abs{D+C}+\partial_t}{D+C-\partial_t}$
%  belongs to $D^* (\widetilde{M}\times\RR)^\Gamma$} 
%\label{sec:belong_to_D}
%% list of sections ends

\bigskip
\noindent
{\bf Acknowledgments.} We are glad to thank Georges Skandalis, 
Stephane Vassout, Charlotte Wahl, Rudolf Zeidler and Vito Felice Zenobi for interesting discussions.
We thank the anonymous referee for a very careful reading of the original
manuscript and for valuable suggestions.
P.P.~thanks the \textit{Projet Alg\`ebres d'Op\'erateurs} of \emph{Institut de
  Math\'ematiques de Jussieu} for hospitality and financial support
while part of this research was carried out; he also thanks
{Mathematisches Institut} and the Courant Research Center ``Higher order structures in
mathematics'' for their hospitality for several visits to G\"ottingen;
the  financial support 
of {\em Ministero dell'Universit\`a e della Ricerca Scientifica} (through the
project ``Spazi di Moduli e Teoria di Lie'') is also gratefully acknowledged.
T.S.~acknowledges the support of the Courant Research Center ``Higher order
structures in mathematics''.

\section{Index and rho classes defined by perturbations}
\label{sec:indexdef}

\subsection{The index homomorphism in L-theory}\label{sub:L-map}
The map $\Ind_\Gamma\colon L_{n+1}(\ZZ\Gamma) \to K_{n+1} (C^*_r\Gamma)$ has
been defined 
by Wahl, building on results of Hilsum-Skandalis \cite{HiSka} and the authors
\cite{PS1}. We briefly describe it. Assume that $n+1$ is even. 
Recall, see for example \cite[Chapter 4]{higson-roeIII}
and the references therein, that an element $x\in
L_{n+1}(\ZZ\Gamma)$ is represented by a quadruple $(W,F,X\times [0,1],u\colon
X\to B\Gamma)$ 
with $W$ a cobordism between two smooth orientable manifolds $\pa_1 W$ and $\pa_2 W$, $X$ a smooth orientable manifold,
$F\colon (W, \pa W) \to (X\times [0,1], \pa(X\times [0,1]))$ a degree one normal map of pairs, $f_1:=F\,|_{\pa_1 W}$ and $f_2:=F\,|_{\pa_2 W}$ oriented homotopy equivalences
and $u\colon X\to B\Gamma$ a classifying map. Let $f=f_1 \disjointunion f_2$ denote the restriction of $F$ to $\pa W$. Consider 
$Z:= W\disjointunion X\times [0,1]$, a manifold with boundary. Let $\mathcal{D}_Z$ be the signature operator on $Z$
with coefficients in the Mishchenko bundle defined by $u\colon X\to B\Gamma$
and $u\circ F\colon W\to B\Gamma$, i.e.~the bundle obtained as pullback of
the $C^*_{r}\Gamma$-module bundle $E\Gamma\times_\Gamma C^*_{r}\Gamma$
over $B\Gamma$. Then, {proceeding as in  \cite{Wahl_higher_rho}}, we can construct a 
smoothing perturbation $\mathcal{C}_f$
of the boundary operator $\mathcal{D}_{\pa Z}$ with the property
that $\mathcal{D}_{\pa Z} + \mathcal{C}_f$ is invertible. This perturbation, a smoothing
version of the original one defined by Hilsum and Skandalis in \cite{HiSka},
 is, first of all, a {self-adjoint} bounded operator 
on the $C^*_r \Gamma$-Hilbert module $\mathcal{E}_{{\rm M}}:= L^2
(Z, \Lambda^* Z\otimes \mathcal{F}_{{\rm M}})$ with $\mathcal{F}_{{\rm M}}$
denoting
the Mishchenko bundle. {Moreover,
%, and this will be useful later on,
it is  an element in} $\Psi^{-\infty}_{C^*_r \Gamma} (Z, \Lambda^* Z\otimes \mathcal{F}_{{\rm M}})$, the smoothing operators in the Mishchenko-Fomenko calculus.
%We denote the
%Hilsum-Skandalis perturbation by  $ \mathcal{A}_f$; this is a bounded operator 
%on the $C^*_r \Gamma$-Hilbert module $\mathcal{E}_{{\rm M}}:= L^2
%(Z, \Lambda^* Z\otimes \mathcal{F}_{{\rm M}})$ with $\mathcal{F}_{{\rm M}}$ denoting
%the Mishchenko bundle.\\
Wahl's perturbation $\mathcal{C}_f$ is an example of what is called, in the literature,
a smoothing trivializing perturbation; more generally a
{\it trivializing perturbation} for $\D$ is a self-adjoint $C^*_r \Gamma$-bounded operator $\C$ with the property that $\D + \C$
is invertible as an unbounded operator on the Hilbert module $\mathcal{E}_{{\rm M}}$.
The original perturbation defined by  Hilsum-Skandalis is an example of such a perturbation, see \cite[Proposition 3.1]{zenobi}
for a proof.
Notice that in this paper only the smoothing trivializing perturbation $\mathcal{C}_f$ will be used; this corresponds to the choice 
$\epsilon\in (0,+\infty)$ in Wahl's treatment.\\
%This perturbation has been slightly modified in \cite{Wahl_higher_rho}. 
%Strictly speaking  $\mathcal{C}_f$ is a trivializing smoothing
%perturbation in the $C^*_r \Gamma$-calculus. However, using the fact that
%$\mathcal{B}^\infty_\Gamma$ is a dense holomorphically closed subalgebra
%one can prove the following
%
%%\begin{remark}
%%The perturbation constructed in \cite{PS1}
%%was made smoothing by composing and precomposing with $\phi(\mathcal{D}_{\pa Z})$, with
%%$\phi$ a bump function equal to $1$ in $0$; we might want to take the perturbation
%%of order $0$ defined using the work of Hilsum-Skandalis, thus without the smoothing
%%term (this might be useful in establishing
%%the cobordism invariance, see below).
%%\end{remark}
%
%\begin{lemma}
%The perturbation $\mathcal{C}_f$, which is a priori a smoothing operator in the 
%$C^*\Gamma$-Mishchenko-Fomenko calculus, 
%is in fact a $\mathcal{B}^\infty_\Gamma$-smoothing
%operator.
%\end{lemma}
We extend Wahl's perturbation
$\mathcal{C}_f$ in the obvious way to the cylinder $\RR\times \pa Z$
(we extend it to be constant in the cylindrical direction) and then use a 
cut-off function in order to graft  this operator to the manifold with cylindrical 
end $Z_\infty$ associated to $Z$. We denote this
global perturbation by $\C_{f,\infty}$;
this is  the global perturbation
chosen by Wahl and it is the one we shall take. (In previous work on higher APS
index theory the global perturbation $\C_{f,\infty}$ was chosen to be $b$-pseudodifferential, see
\cite{MPI} \cite{LPGAFA};
while this choice would simplify some of our arguments in Section  \ref{subsub:coarse-classes},
it would eventually make the proof of our main theorem more involved; this is why we have chosen
the perturbation just explained.)
Proceeding as in 
\cite{LPGAFA} and  \cite[Theorem 10.1]{LLP} one proves that 
there is a well defined index class
associated to  $\mathcal{D}_{\infty}+\mathcal{C}_{f,\infty}$,
with $\D$ denoting the Mishchenko-Fomenko signature operator. The index class
is an element  in $K_{n+1} (C^*_r \Gamma)$. 
See Section \ref{subsub:compatibility} below for further details.
Thus, to 
the  quadruple  $(W,F,X\times [0,1],u\colon X\to B\Gamma)$ we associate 
$\Ind (\mathcal{D}_{\infty}+\mathcal{C}_{f,\infty})\in K_{n+1} (C^*_r \Gamma)$.
%  Moreover, one can prove, proceeding as in the proof of \cite[Proposition 6.4]{LPGAFA},
%  that the index class of 
%  $\mathcal{D}_Z+\mathcal{C}_f^b$ and the index class of  $\mathcal{D}_Z+\mathcal{C}_{f,\infty}$
%  are equal in $K_{n+1} (\mathcal{B}^\infty_\Gamma) =K_{n+1} (C^*_r \Gamma)$.
%  %\footnote{The advantage of taking a lift in the $b$-calculus
%shows itself when we wish to prove higher index formulae for the
%Chern character of the index class.}\\

\begin{theorem}\label{theo:charlotte}
This construction  induces a well defined group homomorphism 
\begin{equation*}
\Ind_\Gamma\colon L_{n+1}(\ZZ\Gamma) \rightarrow %K_{n+1} (\mathcal{B}^\infty_\Gamma)=
K_{n+1} (C^*_r \Gamma)
\end{equation*}
\end{theorem}

\begin{proof}
This is proved by Charlotte Wahl in \cite[Theorem 9.1]{Wahl_higher_rho}.
%The proof is in three steps
%\begin{itemize}
%\item use Wall's realization theorem in order to show that each $x\in L_{n+1}(\Gamma)$ has
%a representative given by a quadruple as above;
%\item show that if two such quadruples are equivalent, then there is a triad in the smooth category
%(smooth manifolds with corners)
%realizing the equivalence;
%\item (cobordism invariance) show that if 
%$(W,F,X\times [0,1],u\colon X\to B\Gamma) \sim (W',F',X'\times [0,1],u'\colon X'\to B\Gamma)$
%through a smooth triad, then 
%$$\Ind ( \mathcal{D}_Z+\mathcal{C}_f^b ) = \Ind ( \mathcal{D}_{Z'}+\mathcal{C}_{f'}^b ) \quad \text{in}\quad
%K_{n+1} (\mathcal{B}^\infty_\Gamma).$$
%Notice that this index class is given by a formal difference of finitely generated projective 
%$\mathcal{B}^\infty_\Gamma$-modules
%\end{itemize}
\end{proof}
We give more information on this index class in Section
\ref{subsub:coarse-classes}.

%\begin{remark}
%$\clubsuit$ From Paolo:\\
%IMPORTANT: we should keep in mind that while Charlotte 
%employs   our perturbation (the one compressed with $\phi (D)$) in order
%to define the map on cycles, she does use the Hilsum-Skandalis perturbation
%(uncompressed) in order
%to prove that it is well defined on $L_{n+1} (\ZZ\Gamma)$. $\clubsuit$$\clubsuit$
%\end{remark}

%\begin{remark}
%Notice that if we take the Connes-Moscovici projection associated to a $b$-parametrix
%for $\mathcal{D}_Z+\mathcal{C}_f^b$, then we  get an element in the K-theory
%of the residual algebra $\Psi^{-\infty,\epsilon}_{\mathcal{B}^\infty_\Gamma}$ of $Z$.
%From a $C^*$-point of view, we know that this algebra embeds into the compact operators
%of the $C^*\Gamma$-Hilbert module given by the $L^2$-sections of the $C^*\Gamma$-M bundle;
%thus, from the $C^*\Gamma$-point of view, the Connes-Moscovici index class 
%does give, by Morita equivalence,
%our  class in $K_{n+1} (C^*\Gamma)$. 
%Here are some questions:\\
%Can we study these equivalences directly ?\\
%In other words, is $\mathfrak{C}^* V$ Morita equivalent to $\mathfrak{C}^*_0 (W\sqcup (V\times [0,1]))$ ?
%(The $0$-subscript means vanishing of the kernel at the boundary.)\\
%Similarly: is $\mathfrak{C}^* V$ Morita equivalent to $\mathfrak{C}^* (\pa W\sqcup \pa(V\times [0,1]))$ ?\\
%What about the smooth versions of these algebras ?\\
%What about the corresponding  $\mathfrak{D}^*$-algebras ?

%\end{remark}

\subsection{Rho classes}\label{sub:coarse-index}
%The index map out of $L_{n+1} (\ZZ\Gamma)$ is constructed using higher 
%index theory in the Mishchenko-Fomenko framework. We wish to frame the construction of the perturbation $\mathcal{C}_f$,
%%and $\mathcal{A}_f$ ,
%and of the resulting  index class, in large scale index theory, also called
%coarse index theory.

In this subsection we first fix a %an odd dimensional 
 $\Gamma$-manifold $\widetilde{V}$ with a free
cocompact action of $\Gamma$ with quotient $V$, a smooth compact
manifold without boundary. We fix a $\Gamma$-invariant metric
on $\widetilde{V}$; we also fix a $\Gamma$-equivariant hermitian 
vector bundle $\widetilde{E}$ on $\widetilde{V}$ with quotient $E$ on $V$. 
We assume the existence of
a $\Gamma$-equivariant Clifford structure on $\widetilde{E}$ and we denote by $D$
the corresponding Dirac type operator on $\widetilde{V}$; this is a $\Gamma$-equivariant
operator. Notice that we {\it do not} employ
    the tilde-notation for the operators on the covering.
    We denote by $\D$ the induced operator in the Mishchenko-Fomenko 
calculus.

Recall the main players in the Higson-Roe surgery sequence.
We have the $C^*$-algebra
$C^*(\tilde V)^\Gamma$, called  Roe algebra, obtained as the closure  of the
locally compact $\Gamma$-invariant finite propagation operators,  and  we have
the $C^*$-algebra 
$D^*(\tilde V)^\Gamma$, obtained as the closure of  $\Gamma$-invariant pseudolocal finite propagation
operators. 
Recall also that $C^*(\tilde V)^\Gamma$ is an ideal in $D^*(\tilde V)^\Gamma$.
We refer to the companion
paper \cite{PS-Stolz} for the precise definitions and for the notation we adopt.
One of the extra subtleties is that one has to stabilize the bundles by
tensoring with $l^2(\naturals)$ for the definition in particular of
$D^*(\tilde V)^\Gamma$ and let operators act on one corner $V\tensor \complexs
e_1\subset V\tensor l^2(\naturals)$. See for example  \cite[Section 1.2]{PS-Stolz}.
We will suppress this throughout in the
notation.

%\smallskip
%\noindent
%We begin by assuming that $\widetilde{V}$ is
%odd dimensional 

\subsubsection{Operators on the covering}

Recall from \cite{PS-Stolz} that given a cocompact Galois covering
$\widetilde{V}\to V$ there is an isomorphism 
\begin{equation}\label{c-star-iso}
 \KK (\mathcal{E}_{{\rm M}})\cong C^* (\widetilde{V})^\Gamma \end{equation}
 where we recall that $\mathcal{E}_{{\rm M}}$ stands for the 
 Mishchenko $C^*_r \Gamma$-Hilbert module
 $L^2 (V,E\otimes \mathcal{F}_{{\rm M}})$, with $\mathcal{F}_{{\rm M}}=\widetilde{V}\times_\Gamma C^*_r \Gamma$, the Mishchenko bundle.
  Consider now $V$ and $\widetilde{V}$ as above. Let 
 $\mathcal{C}\in \Psi^{-\infty}_{C^*_r \Gamma} (V, E\otimes \mathcal{F}_{{\rm M}})$
 be a   smoothing trivializing perturbation for a Dirac type operator $\D
 \in {\rm Diff}^1_{C^*_r \Gamma} (V, E\otimes \mathcal{F}_{{\rm M}})$.
  Obviously, from the Mishchenko-Fomenko calculus, we have that
 $\Psi^{-\infty}_{C^*_r \Gamma} (V,E\otimes \mathcal{F}_{{\rm M}})
 \subset \KK (\mathcal{E}_{{\rm M}})$.
 Using \eqref{c-star-iso}
we obtain immediately that $\mathcal{C}$ defines an element
  $C$ in $C^* (\widetilde{V})^\Gamma$. 
   
   {In what follows we have two goals in mind: on the one hand 
 we wish to generalize the assignment of the operator $C$  to any self-adjoint 
 bounded operator $\C$ (not necessarily
 a smoothing operator in the Mishchenko-Fomenko calculus); 
 on the other hand we wish to give a precise definition for such a $C$.
 To this end let} $\pi\colon C^*_r \Gamma\to \mathcal{B} (\ell^2 (\Gamma))$ be the left
regular representation. 
%We set $H=\ell^2 (\Gamma)$
Recall e.g.~from \cite{Schick-ell2} 
that tensoring with $\pi$ (a faithful representation) induces an isomorphism
of right Hilbert $\Gamma$-modules
\begin{equation}\label{tensor}
 \mathcal{E}_{{\rm M}}\otimes_\pi \ell^2 (\Gamma)\rightarrow L^2 (\widetilde{V}, \Lambda^* 
\widetilde{V}) .
\end{equation}
Associating to each $\mathcal{C}\in \BB (\mathcal{E}_{{\rm M}})$
the operator $\mathcal{C}\otimes_\pi \Id_{\ell^2 (\Gamma)}$ in
$\mathcal{B}(\mathcal{E}_{{\rm M}}\otimes_\pi \ell^2 (\Gamma))$
  we define a homomorphism of $C^*$-algebras
$$ \BB (\mathcal{E}_{{\rm M}})\ni \mathcal{C}\mapsto
\mathcal{C}\otimes_\pi \Id_{\ell^2 (\Gamma)}\in \mathcal{B}(\mathcal{E}_{{\rm
    M}}\otimes_\pi \ell^2 (\Gamma)) .$$
 Notice that  the right hand side
 %$C_\pi\in 
%\mathcal{B} (L^2 (\widetilde{V}, \Lambda^* 
%\widetilde{V}))$ 
is $\Gamma$-equivariant, given that
$\mathcal{C}$ is $C^*_r \Gamma$-linear.
Thus, conjugating with \eqref{tensor}, we obtain a $C^*$-homomorphism
$$ L_\pi\colon  \BB (\mathcal{E}_{{\rm M}})\ni \mathcal{C}\mapsto
C_\pi \in 
\mathcal{B} (L^2 (\widetilde{V}, \Lambda^* 
\widetilde{V}))^\Gamma$$

\begin{proposition}\label{prop:hs-mult}
$L_\pi$ sends 
$\KK (\mathcal{E}_{{\rm M}})$ isomorphically onto $C^* (\tilde{V})^\Gamma$.
Moreover, it sends  $ \BB (\mathcal{E}_{{\rm M}})$ into \\
$\mathfrak{M} ( C^* (\tilde{V})^\Gamma)$, the multiplier algebra
of $C^* (\tilde{V})^\Gamma$.
\end{proposition}

\begin{proof}
$L_\pi$ is injective, given that $\pi$ is faithful.
The statement about $L_\pi (\KK (\mathcal{E}_{{\rm M}}))$  follows by looking
at the 
image through $L_\pi$ of a dense set in $\KK (\mathcal{E}_{{\rm M}})$. 
We can choose, for example, 
$\Psi^{-\infty}_{\CC\Gamma}$. If $A\in \Psi^{-\infty}_{\CC\Gamma}$ then
$L_\pi (A)$ is nothing but the associated  $\Gamma$-compactly 
supported smoothing operator on the covering (see \cite[Proposition 6]{LottI})
and these operators are dense in $C^* (\tilde{V})^\Gamma$. Thus 
$L_\pi$ sends 
$\KK (\mathcal{E}_{{\rm M}})$ isomorphically onto $C^* (\tilde{V})^\Gamma$.
The second statement is classical: 
indeed, Kasparov has proved, see \cite{WO}, that
$ \BB (\mathcal{E}_{{\rm M}})$ {\it is} the multiplier algebra of $\KK (\mathcal{E}_{{\rm M}})$.
The statement now follows from general arguments; indeed, 
given an injective representation $\phi$ of a $C^*$-algebra $K$ into the
bounded operators
of a Hilbert space $H$, then $\phi$ extends uniquely to a representation
$\tilde{\phi}$ of the multiplier algebra of $K$ and $\tilde{\phi}(\mathfrak{M}({K}))=\mathfrak{M}
(\phi (K))$.
\end{proof}

%OLD COMMENT
% (from Thomas)
% add references and make ptoof more precise 
{We are now ready to define precisely the bounded operator $C$, on the covering
$\widetilde{V}$,
corresponding to a self-adjoint operator
$\C\in \BB (\mathcal{E}_{{\rm M}})$.
We simply  set
\begin{equation}\label{ell-pi}
C:= L_\pi (\mathcal{C})\,.
\end{equation}
We know that this is an element in
$\mathfrak{M} ( C^* (\tilde{V})^\Gamma)$, the multiplier algebra
of $C^* (\tilde{V})^\Gamma$.
If $\C$ is, in addition,  smoothing, i.e. 
$\C\in
\Psi^{-\infty}_{C^*_r \Gamma} (V, E\otimes \mathcal{F}_{{\rm M}})$,
then $C\in C^* (\tilde{V})^\Gamma$.}

  \begin{remark}
  If we apply all this to the smoothing
  trivializing perturbation $\mathcal{C}_f$ defined by a homotopy equivalence
  $f\colon M\to V$, { we immediately see that $\mathcal{C}_f$  defines a 
  perturbation }
  $C_f:= L_\pi (\C_f)$ in $C^* (\widetilde{Z})^\Gamma$, with $Z=M\cup (-V)$ and 
  $\widetilde{Z}$ the Galois covering defined by $f\circ u\cup (-u): M\cup (-V)\to B\Gamma$, $u$ denoting a classifying map for the universal cover of $V$.
  \end{remark}

\subsubsection{The  rho-class associated to a perturbation}

Let $\D\in {\rm Diff}^1_{C^*_r \Gamma} (V, E\otimes \mathcal{F}_{{\rm M}})$ be a 
Dirac type operator as above. We first assume $V$ and therefore $\widetilde{V}$
to be odd dimensional. Assume that there exists a %smoothing 
{trivializing
perturbation $\C$ for $\D$.
We recall that this means that $\C$ is a self-adjoint bounded operator on
 $\mathcal{E}_{{\rm M}}$, the 
 Mishchenko $C^*_r \Gamma$-Hilbert module, with the property that the self-adjoint regular 
 operator
$\D+\C$ is invertible.}  Following \cite{MPI}, it is proved in \cite{LP03}
that this 
is true if and only if the index class $\Ind (\D)\in K_1 (C^*_r\Gamma)$
vanishes.

For the proof of the following proposition  see, e.g., the proof of 
\cite[Lemma 2.1]{LLP}.

\begin{proposition}
Let $D$ be the Dirac operator on $\widetilde{V}$ corresponding to $\D$.
Let $C:= L_\pi (\C)$.  Then 
$D+C$ %\quad\quad \text{ and }\quad \quad D+A_f$$ are 
is self-adjoint
and $L^2$-invertible.
\end{proposition}

%Next we recall a technical but fundamental result of Higson-Roe. We go back
%to a general closed orientable manifold $V$.
%Consider the   ideal $C^* (\widetilde{V})^\Gamma$ in $D^* (\widetilde{V})^\Gamma$.
%Recall form \cite[Definition 5.3]{higson-roeI} the notion of unbounded self-adjoint operator $D$ analitically
%controlled over $(D^* (\widetilde{V})^\Gamma\,,\, C^* (\widetilde{V})^\Gamma)$.
%This means that $D$ is an unbounded operator on the Hilbert module $H$
%appearing in the definition of $D^* (\widetilde{V})^\Gamma$ and, moroever,

%\begin{itemize}
%\item the resolvent of $D$ is in $C^* (\widetilde{V})^\Gamma$ 
%\item  $D (1+D^2)^{-1/2}$ is in $D^* (\widetilde{V})^\Gamma$
%\end{itemize}

%Equivalently:

%\begin{itemize}
%\item $f(D)\in C^* (\widetilde{V})^\Gamma$ $\forall f\in C_0 (\RR)$
%\item $f(D)\in D^* (\widetilde{V})^\Gamma$ for every $f\in C([-\infty,\infty])$.
%\end{itemize}

Next we recall a result of Higson-Roe, implicitly proved in \cite[Proposition
5.9]{higson-roeI}.

\begin{proposition}\label{prop:hr-controlled-sharp}
Let $D$ be a self-adjoint unbounded operator on $H:=
L^2(\widetilde{V},\widetilde{E})$, as above. 
Let $\mathfrak{A}$ be a $C^*$-algebra in $\mathcal{B}(H)$ and let $\mathfrak{J}$
be an ideal in $\mathfrak{A}$. Let $\mathfrak{M}$ be the multiplier algebra
of $\mathfrak{J}$.
Assume that $S$ is a self-adjoint operator in $\mathfrak{M}$, that the resolvent of
$D$ is in $\mathfrak{J}$ and that $D (1+D^2)^{-1/2}$ is in $\mathfrak{A}$. Then
the resolvent of $D+S$ is in $\mathfrak{J}$
and $(D+S) (1+(D+S)^2)^{-1/2}\in \mathfrak{A}$. Consequently, 
\begin{itemize}
\item $f(D+S)\in \mathfrak{J}$, $\;\;\forall f\in C_0 (\RR)$
\item $f(D+S)\in \mathfrak{A}$,  $\;\;\forall f\in C([-\infty,\infty])$.
\end{itemize}

\end{proposition}

\begin{proposition}
Let $\D$ and $D$ as above; let $\C$ be a %smoothing 
trivializing perturbation
for $\D$ and let $C:= L_\pi (\C)\in \mathfrak{M} (C^* (\widetilde{V})^\Gamma)$ be the corresponding perturbation of $D$ on the
covering.
Then 
$$\frac{D+C}{|D+C|}\text{ is an element in }D^* (\widetilde{V})^\Gamma\,.$$
\end{proposition}

\begin{proof}
Choose $\mathfrak{A}=D^* (\widetilde{V})^\Gamma$ and $\mathfrak{J}=
C^* (\widetilde{V})^\Gamma$. Then we can apply Proposition
\ref{prop:hr-controlled-sharp}, given that $C\in \mathfrak{M} (C^* (\widetilde{V})^\Gamma)$ (see
Proposition \ref{prop:hs-mult}).
We thus obtain that $\chi(D+C)$  is an element in $D^* (\widetilde{V})^\Gamma$ for any chopping
function $\chi$;
choosing $\chi$ equal to $\pm 1$ on the spectrum of $D+C$ we are done.
\end{proof}

%Of course, the same proof would have established that
%$(D+A)/|(D+A)|$  is an element in $D^* (\widetilde{V})^\Gamma$
%assuming 
% only that $A\in \mathfrak{M} (C^* (\widetilde{V})^\Gamma)$ (and 
% of course that 
%$D+A$ is $L^2$-invertible).
This brings us to the definition
of rho-classes.

\begin{definition}\label{def:rho-perturbed}
Let  $\widetilde{V}\to V$ a $\Gamma$-covering
of a smooth compact orientable  odd dimensional manifold $V$ without boundary.
Let $D$ be a $\Gamma$-equivariant Dirac operator on $\widetilde{V}$
acting on the sections of a $\Gamma$-equivariant bundle $\widetilde{E}$.
Let $\mathfrak{M}$ be the multiplier algebra of $C^* (\widetilde{V})^\Gamma$.
For any self-adjoint operator $A\in \mathfrak{M}$ with the property
that $D+A$ is $L^2$-invertible we can consider the operator $(D+A)/|D+A|$.
Then, proceeding as above we have that 
$$\frac{D+A}{|D+A|}\text{ is an element in }D^* (\widetilde{V})^\Gamma\,.$$
The rho-class associated to $D$ and to the trivializing perturbation 
$A$ is, by definition, the idempotent defined by the involution $(D+A)/|(D+A)|$:
 \begin{equation}\label{rho-odd}
 \rho(D+A):= [\frac{1}{2}\left( \frac{D+A}{|D+A|} + 1 \right)]\;\in\; K_0 (D^* (\widetilde{V})^\Gamma)\,.
 \end{equation}
In the even dimensional case, with a $\ZZ/2$-graded bundle 
$\widetilde{E}=\widetilde{E}^+ \oplus \widetilde{E}^- $ 
we can proceed analogously, once we have  a
$\ZZ/2$-graded  trivializing self-adjoint perturbation 
$A\in \mathfrak{M} (C^* (\widetilde{V})^\Gamma)$, viz.
\begin{equation}\label{z2-pertubed}
   D+A= \begin{pmatrix}
      0 & D^- + A^-\\ D^+ + A^+ & 0
    \end{pmatrix}
   \end{equation}
A necessary and sufficient
 condition for such a perturbation to exist is that $\Ind (\D)=0$ in $K_0
 (C^*_r \Gamma)$, see \cite{MPII,LP03}.
 In order to define the rho class associated to \eqref{z2-pertubed} we consider the space, call it $J$, of  $\Gamma$-equivariant
bundle isometries from $\widetilde{E}^-$ to
$\widetilde{E}^+$ (more precisely,
on their stabilizations,  obtained by tensoring the two bundles with $l^2(\naturals)$, see the beginning of Section \ref{sub:coarse-index}; this makes sure
that the bundles are trivial). 
Any element $u$ in $J$ induces in a natural way an  
isometry $U\colon L^2 (\widetilde{V}, \widetilde{E}^-)
\to L^2 (\widetilde{V}, \widetilde{E}^+)$. It is obvious that  $U$ commutes with the action of $\Gamma$ and that
it covers the identity in the $D^*$ sense, see for example \cite[Section 1]{PS-Stolz}.
 Thus we can set
 \begin{equation}\label{rho-even}
\rho(D+A):= [U\chi(D + A)_+]\in K_1 (D^* (\widetilde{V})^\Gamma)
\end{equation}
with $\chi$ an odd  chopping function equal to the sign function on the  spectrum of
$D+A$. 
% Let $k$ be equal to the rank of $E^+$.
Since $J$ is isomorphic, once we stabilize,  to the space of $U(H)$-valued functions 
on the quotient $V=\widetilde{V}/\Gamma$ ($H$ a separable Hilbert space), and since $U(H)$ is contractible, we
see immediately that also $J$ is contractible. Thus, by homotopy invariance of
K-theory, the right hand side in \eqref{rho-even} 
does not depend on the choice of $u$ in $J$. 

\begin{remark} {In our previous paper \cite{PS-Stolz} we allowed for
    an arbitrary $\Gamma$-equivariant isometry $U\colon L^2 (\widetilde{V},
    \widetilde{E}^-) \to L^2 (\widetilde{V}, \widetilde{E}^+)$ covering the
    identity in the $D^*$-sense, see the discussion leading to Definition
    1.11.  \emph{This is not correct} for we could then change at will the
    class on the right hand side of \eqref{rho-even} by composing on the left
    with an arbitrary $\Gamma$-equivariant unitary $V\in U(L^2 (\widetilde{V},
    \widetilde{E}^-))$, covering the identity in the $D^*$-sense.} Using the
  more specific choice of $U$ as above, everything in \cite{PS-Stolz} goes
  through as stated there.
\end{remark}

Finally, if we consider the canonical map $u\colon \widetilde{V}\to E\Gamma$
then we define
\begin{equation}\label{real-rho}
\rho_{\Gamma} (D+A) := u_* \rho(D+A) \;\in\; K_* (D^*_\Gamma).
\end{equation}
\end{definition}

\begin{remark}
The rho classes do depend, in general,  on the choice of the trivializing
perturbation 
$A$. This will be clear from our delocalized APS index theorem
\ref{theo:k-theory-deloc} for perturbed
operators.

\end{remark}
\begin{remark}
Note that $K_*(D^*_\Gamma)$, by the Baum-Connes conjecture, is expected to
vanish for torsion-free groups, but is often non-zero otherwise. Therefore,
the universal $\rho$-class $\rho_\Gamma$ is of interest essentially only if
$\Gamma$ is a group with non-trivial torsion.
\end{remark}

\subsubsection{Fundamental examples of rho classes}

We present two fundamental examples of $\rho$-classes. As we shall see later these
examples enjoy strong stability properties with respect to the trivializing perturbation.

\begin{definition}\label{example:rho-he}
Let $(V,g)$ be an oriented smooth Riemannian manifold without boundary
with fundamental group $\Gamma$. Let $u\colon V\to B\Gamma$ be the
classifying map for the universal cover of $V$.
Let  $(M,h)$ be another oriented Riemannian
manifold without boundary and assume that $M\xrightarrow{f}V$ is an oriented
homotopy equivalence.
We consider $Z= M\disjointunion (-V)$ with the obvious classifying map $u_Z
  \colon Z\to B\Gamma$ 
  induced  by $u$ and by $u\circ f$. We then obtain
  $$\rho (D + C_f)\in K_{n+1} (D^* (\widetilde{Z})^\Gamma)\quad\text{and}\quad
   \rho_\Gamma (D + C_f)\in K_{n+1} (D^*_\Gamma)$$
  with $D$  equal to the signature operator
  on the covering $\widetilde{Z}:= u_Z^* E\Gamma$ and $C_f$ the smoothing
  trivializing perturbation defined by the homotopy equivalence $f$. {We choose
  Wahl's version of this perturbation}. Notice that
  there is an obvious $\Gamma$-equivariant map $\tilde{\phi}\colon \widetilde{Z}\to \widetilde{V}$,
  $\tilde{\phi}:=\tilde{f}\cup \id_{(-\widetilde{V})}$.
  We set 
   \begin{equation}\label{def-of-rho-Gamma(f)}
  \rho  (f):=\tilde{\phi}_* \rho (D+C_f)\in K_{n+1}(D^*
  (\widetilde{V})^\Gamma);\qquad   \rho_\Gamma (f):=\rho_\Gamma(D+C_f)\in
  K_{n+1}(D^*_\Gamma).
  \end{equation}
Observe  that, by functoriality,  $\rho_\Gamma (f)= \tilde{u}_*   \rho  (f)$.
\end{definition}

\begin{proof}
  We have to argue why $\rho (f)$ and $\rho_\Gamma(f)$ do not depend on the choices
  involved.
The signature operator on $Z$ depends on the choice of the Riemannian metrics
 on $M$ and $V$.
 The perturbation $\C_f$ depends on several choices (see \cite{Wahl_higher_rho} for
 details); it also depends on the choice of the metric on $M$ and $V$.
  Wahl proves (compare \cite[Section 4]{Wahl_higher_rho}) that two different
  choices can be joined by a path of invertible operators
  $D(t)+ B(t)$; thus, according to (an easy extension of) 
  Proposition \ref{prop:stab-index-class} below,
  the rho-class $\rho (D + C_f)\in K_{n+1} (D^* (\widetilde{Z})^\Gamma)$ is independent
  of the choice of the Riemannian metrics and of
  the choices we have made in defining the trivializing perturbation
  $C_f$. Consequently, also $\rho (f)$ and $\rho_\Gamma(f)$ are independent of these choices.
\end{proof}

  \begin{example}\label{example:LLP-rho}
Let $(M,g)$ be a closed oriented Riemannian manifold and 
$\widetilde{M}\to M$ a $\Gamma$-cover. Let $2m$ or $2m+1$ be the
dimension of $M$, depending whether $M$ is even or odd dimensional.
We assume that the Laplacian on differential forms on the covering is
$L^2$-invertible in degree $m$. We shall say briefly that the Laplacian on
the covering is invertible in middle degree.
By the homotopy invariance of $L^2$-Betti numbers and of the Novikov-Shubin
invariants this is a homotopy invariant condition.
Let $D$ be the signature operator on $\widetilde{M}$ and let $\D$ be the
associated Mishchenko-Fomenko operator
As explained in \cite{LPAGAG} there is then a class of smoothing 
trivializing perturbations
$\mathcal{S}$, that were named {\it symmetric} there, and that enjoy strong
stability properties, as we shall explain in a moment.
We define the rho-class of a manifold $M$ satisfying the above condition 
as the $K_* (D^* (\widetilde{M})^\Gamma)$-class
$$\rho (D+\mathcal{S}) \quad\text{with}\quad \mathcal{S}\text{ a symmetric trivializing perturbation.}
$$
with  $*=\dim M + 1$.
We also define $\rho_\Gamma (D+\mathcal{S})=\tilde{u}_* (\rho
(D+\mathcal{S}))\in K_* (D^*_\Gamma)$.
We show in Remark \ref{rem:symmetric} that
\begin{equation}\label{invariance-symm}
\rho_\Gamma (D+\mathcal{S})=\rho_\Gamma (D+\mathcal{S}^\prime) \in  K_*
(D^*_\Gamma) \quad \text{if } \mathcal{S},  
\mathcal{S}^\prime \text{ are  symmetric trivializing perturbations.}
\end{equation}
This brings us to the following definition:\\
{\it Let $\widetilde{M}\to M$ a $\Gamma$-cover and assume 
 that the Laplacian on differential forms on  $\widetilde{M}$  is invertible in middle degree.
 Let $D$ be the signature operator on  $\widetilde{M}$.
We  then define  $\rho_\Gamma (\widetilde{M})\in K_* (D^*_\Gamma)$ as
$$\rho_\Gamma (\widetilde{M}):= \rho_\Gamma (D+\mathcal{S}) \text{ for any  symmetric trivializing perturbation }
\mathcal{S}.$$ 
}
The notation is justified because, as in the previous example,
 the right hand side is also independent  of the choice of  metric.
Working a bit harder one can extend these results to odd
dimensional
Galois coverings $\widetilde{M}\to M$, $\dim M=2m+1$, with the property 
that in degree $m$ the reduced and unreduced cohomology with values 
in the local system defined by the Mishchenko-Fomenko bundle are equal.
See \cite{LLP,Wahl-K-theory}. 
%Note, finally, that with some extra care it is possible
%to sharpen these results to $\rho (D+\mathcal{S}).
  \end{example}

\subsection{Index classes}\label{subsub:coarse-classes}

The index map $\Ind_\Gamma\colon L_{n+1} (\ZZ\Gamma)\to K_{n+1} (C^*_r \Gamma)$ is constructed using APS higher 
index theory in the Mishchenko-Fomenko framework. {We wish to frame the construction %$of the perturbation $\mathcal{C}_f$,
%and $\mathcal{A}_f$ ,
%and 
of  APS index classes on manifolds with boundary and with cylindrical ends in coarse index theory.}

\subsubsection{Coarse index classes on manifolds with cylindrical ends}\label{subsub:cylindrical-classes}

We will now consider manifolds with boundary and with cylindrical ends. First we recall the basic notation.

\begin{notation}
Let $(W,g)$ be a  Riemannian manifold  with boundary $M:=\boundary
W$.
 We shall always assume $g$ to have
product structure near the boundary. We denote 
  by $W_\infty$ the manifold  $W$ with an 
  infinite semi-cylinder $[0,\infty)\times \pa M$ attached to the boundary. 
  %$W_\infty$ can also be thought of
  %as the $b$-manifold associated to  $W$ and we shall freely do so when
  %we employ tools from Melrose' $b$-geometry.\\
  If $\widetilde{W}$ is a $\Gamma$-covering of $W$, then we similarly denote by
$\widetilde{W}_\infty$ the manifold obtained from $\widetilde{W}$ by attaching an 
  infinite semi-cylinder.\\
  We will denote by $P_0$ the multiplication operator on $\widetilde{W}_\infty$
  defined by the characteristic
  function of the subset $[0,\infty)\times \pa \widetilde{W}$ in  $\widetilde{W}_\infty$; similarly,
  we denote by $P_R$ the multiplication operator defined by the characteristic
  function of  the subset $[R,\infty)\times \pa \widetilde{W}$ in $\widetilde{W}_\infty$.
  We have similar operators defined in $\RR\times \pa \widetilde{W}$ and with a small
  abuse of notation we employ the same symbols.\\
   We assume $W$ to be even dimensional.\\
  We consider a $\Gamma$-equivariant Dirac type operator $D$ on $\widetilde{W}$,
  acting on the sections of a $\Gamma$-equivariant $\ZZ/2$-graded Hermitian 
  vector bundle $\widetilde{E}$.
   We denote the boundary operator by $D_{\pa}$ and the natural extension
  of $D$ to   $\widetilde{W}_\infty$ by $D_\infty$. Notice once again  that we
  do not employ the tilde-notation for the operators on the covering.\\
    We denote with $\D$, $\D_\partial$
    and $\D_\infty$ the corresponding operators in the 
    Mishchenko-Fomenko calculus. 
    %(In order to lighten the notation, we shall not make
    %a distiction between operators in the $C^*_r (\Gamma)$ or the $\B^\infty_\Gamma$
    %Mishchenko-Fomenko calculi.)\\
    We adopt the Clifford and grading conventions of \cite{Wahl_higher_rho}.
 \end{notation}
    
     %  (Because of bordism invariance for index classes, see
    % \cite{LP03},
    % the index class of $\D_\partial$ vanishes in $K_1 (C^*_r \Gamma)$; thus,
    % according to the general result proved in 
    % \cite{LP03,MPI}, we see that such a perturbation exists.)
    { We consider $C_{\pa}$, a norm limit of $\Gamma$-equivariant self-adjoint bounded finite propagation  operators on $L^2 (\partial \widetilde{W})$
    such that  
    $D_\partial + C_{\pa} $ is $L^2$-invertible. For example $C_{\pa}$
    is equal to   $L_\pi (\mathcal{C}_{\pa})$, with $\mathcal{C}_{\pa}$  a 
    smoothing trivializing perturbation for
    $\D_\partial$.}
    %It is useful that we choose here the {\it smoothing} perturbation for
    %later arguments.
    We now define a global perturbation on $\widetilde{W}_\infty$.
    
\begin{definition}\label{def:global-perturbation}
 We   extend  $C_{\pa}$  on
    $\RR\times \pa\widetilde{W}$ to be constant in the $\RR$ direction.
    We fix an inward
      collar on $\widetilde{W}$ diffeomorphic to $[-1,0]\times \widetilde{W}$; we consider
      $\widetilde{W}_\infty$; then 
    using a cutoff function equal to 1 on $[0,\infty)$ and equal to $0$ at
    $-1$,  we  graft
    this operator to a bounded operator on $\widetilde{W}_\infty$, denoted
    $C_\infty^+$, acting from the sections of $\widetilde{E}^+$ to the sections 
    of $\widetilde{E}^-$. We set  $C_\infty^- := (C_\infty^+)^*$, the formal
    adjoint of $C_\infty^+$, and we consider 
    $$C_\infty :=  \begin{pmatrix}
      0 & C_\infty^-\\ C_\infty^+ & 0
    \end{pmatrix} \;.$$
    \end{definition}
%    We can also extend in the same way $C_{\pa}$ to an operator on 
%    $\widetilde{W}$ (using a cut-off function equal to 1 in a collar neighbourhood
%    of the boundary and equal to 0 in the complement of a larger collar neighbourhood).\\
    We consider the unbounded operator $D_\infty + C_\infty$, which is odd
    with respect to the 
    $\ZZ/2$-grading induced by $\widetilde{E}$.  
    The operator $C_\partial$  is, by assumption,
     a norm limit of finite propagation operators. Therefore  
    $C_\infty$ also has this property. 
    %However, as we essentially
    %tensor $C_\partial$ with the identity on $L^2(\reals)$ in order to obtain
    %$C_{\infty}$, the latter operator fails to be locally compact
    %or pseudolocal and therefore does not belong to $C^*
    %(\widetilde{W}_\infty)^\Gamma$ or $D^*
    %(\widetilde{W}_\infty)^\Gamma$. We shall now deal with this complication.
%      We could certainly define in an analogous way
%      a bounded operator $\mathcal{C}_\infty$
%    on the Mishchenko Hilbert module 
%    $\mathcal{E}_{{\rm M}}:=L^2 (W_\infty, E\otimes \mathcal{F}_{{\rm M}})$
%    and it would follows from naturality that
%    $$C_\infty = L_\pi (\mathcal{C}_\infty)$$
%    with $L_\pi: \BB (\mathcal{E}_{{\rm M}})\to \mathcal{B} (L^2 (\widetilde{W}_\infty, 
%    \widetilde{E})$ the $C^*$-homomorphism induced by the regular representation.
%    However, this would only tell us that $C_\infty \in \mathfrak{M} (C^* (
%    \widetilde{W}\subset \widetilde{W}_\infty)^\Gamma)$; indeed, we have already observed
%    in our previous paper that $\KK (\mathcal{E}_{{\rm M}})\simeq C^* (
%    \widetilde{W}\subset \widetilde{W}_\infty)^\Gamma)$ and we know that, in general,
%    $\BB (\mathcal{E}_{{\rm M}})= \mathfrak{M} (\KK (\mathcal{E}_{{\rm
%    M}}))$.

 More generally, let $B_\infty$ be a bounded self-adjoint  odd
 operator on $W_\infty$ with the following two properties:
 \begin{equation}\label{def-of-B}
B_\infty\;\text{is {norm limit of  finite propagation
    $\Gamma$-operators} and}\;\;
  P_0B_\infty P_0- P_0C_\infty P_0\in \relC{\tilde W_\infty}{\tilde
 W}^\Gamma.
 \end{equation}
Notice that $C_\infty$ has these properties.
 
    \begin{lemma}
    The operators $B_\infty$ belongs to the multiplier algebra
    $\mathfrak{M} (C^* ( \widetilde{W}_\infty)^\Gamma)$. 
    \end{lemma}
    
    \begin{proof}
      Indeed, as is well known, {\emph{every}  operator which
      is norm limit of  $\Gamma$-equivariant finite
      propagation operators} is a multiplier of
      $C^*(\widetilde{W}_\infty)^\Gamma$. {We give for completness
      the easy argument:}
if  $A$ is such an operator with propagation $R$ and if $\phi$ is a
      compactly supported function, then $\phi A=\phi A\psi$ for every
      compactly supported $\psi$ which is equal to $1$ on the $R$-neighborhood
      of the support of $\phi$. Consequently, if $E\in C^* ( \widetilde{W}_\infty)^\Gamma$, $\phi A E=\phi
      A\psi E$ is compact given that $\psi E$ is compact. Passing to norm
      limits, the general statement follows.   \end{proof}
    
    We can now apply the Higson-Roe result, {as stated in Proposition 
    \ref{prop:hr-controlled-sharp}}, with 
    $\mathfrak{A}=D^* (\widetilde{W}_\infty)^\Gamma$,  $\mathfrak{J}=
C^* (\widetilde{W}_\infty)^\Gamma$  and $S=B_\infty\in \mathfrak{M} (C^* (\widetilde{W}_\infty)^\Gamma)$,
obtaining 
\begin{lemma}\label{lem:ancon}
    For every $f\in C_0(\reals)$ we get
    $f(D_\infty+B_\infty)\in C^*(\tilde
    W_\infty)^\Gamma$. If $f\in C([-\infty,\infty])$ then 
    $f(D_\infty+B_\infty)\in D^*
    (\widetilde{W}_\infty)^\Gamma$. 
    \end{lemma}
    
%    \begin{remark}
%    We could have chosen the $b$-smoothing operator $\C^b$
%    associated to $\C_{\pa}$, see \cite[Lemma 9]{MPI},
%    and then its image under $L_\pi$. The resulting $b$-smoothing operator, $C^b$,
%    would have been
%    in $C^* (\widetilde{W}_\infty)^\Gamma$. However, it would have been more difficult
%    to prove the delocalized APS index theorem we are aiming at.
%    \end{remark}
%    
      From Lemma \ref{lem:ancon} we learn that if $\chi$ is a chopping
      function then $\chi (D_\infty + B_\infty)$ is an element in $D^*
    (\widetilde{W}_\infty)^\Gamma$ 
    and an involution modulo $C^* (\widetilde{W}_\infty)^\Gamma$. 
    %Exactly the same reasoning applies to the Hilsum-Skandalis perturbation $A_f$
   However, more is
    true.
    
    \begin{proposition}\label{prop:key}
  If %$\C_{\pa}$ is $C^*_r \Gamma$-compact, if  
   $C_{\pa}$ and  $C_\infty$  are as above,
     if $B_\infty$ is as in \eqref{def-of-B} and $\chi$ is a chopping
    function  equal to the sign function on the spectrum of $D_{\pa} + C_{\pa}$,    then 
    $\chi (D_\infty + B_\infty)$ is an involution modulo $C^* (\widetilde{W}\subset
    \widetilde{W}_\infty)^\Gamma$.
    \end{proposition}
  
  \begin{proof}
Basically, we prove this by comparing $\chi(D_\infty+B_\infty)$ to the 
corresponding
operator, $D_{{\rm cyl}} + C_{\pa}\otimes \Id_{\RR}$, on the the two-sided cylinder
$ \partial \tilde W\times\RR$.

\bigskip
We start with some general comparison results. For simplification, we use the
following notation:
\begin{itemize}
\item We use $D$ instead of $D_\infty$, and $\bar D$ for the corresponding
  (translation invariant) operator on $\reals\times \boundary \tilde W$.
\item We use $C$ instead of $C_\infty$ and $\bar C$ for the translation invariant perturbation of $\bar D$
  on $\reals\times \boundary W$; note that then $\bar D+\bar C$ is
  invertible.
  \item We use $B$ instead of $B_\infty$.
\item We choose a cutoff function $\chi$ which is equal to $\pm 1$ on the
  spectrum of $\bar D+\bar C$, setting $e:=1-\chi^2$ then $e(\bar D+\bar
  C)=0$. 
  \item {The subspace $[0,\infty)\times \boundary \tilde W$ is contained
  isometrically in $\tilde W_\infty$ and in $\reals\times\boundary\tilde W$;
  {\it we freely identify sections supported on this part of both manifolds}. 
  We employ the multiplication
  operator
  $P_R$ for $R\ge 0$; 
  we remark that this is nothing but the orthogonal
  projection onto the sections supported on
  $[R,\infty)\times \boundary \tilde W$. Once again,  with a small
  abuse of notation, we employ the same symbols for this operator
     in $\widetilde{W}_\infty$.
  and in $\RR\times \pa \widetilde{W}$. This way, in particular we consider an operator like $P_0f(\bar
D)P_R$ as acting on $L^2(\tilde W_\infty)$ (instead of $L^2(\boundary\tilde
W\times \reals)$). }
\end{itemize}
%====================================
%{$\bullet$ $\longrightarrow$} 
%In the following, we use the isometric embeddings $\boundary \tilde W\times
%[0,\infty)\into W_\infty$ and $\boundary \tilde W\times [0,\infty)\into \boundary
%W\times \reals$ to identify for $R\ge 0$ the subspaces $P_R L^2(\boundary
%\tilde W\times 
%\reals)$ with $P_RL^2(\tilde W_\infty)$ (both isometric to
%$P_RL^2(\boundary\tilde W\times [0,\infty))=L^2(\boundary\tilde W\times
%[R,\infty))$). 

\begin{lemma}\label{lem:fD_loc}
  For $f\in C_0(\reals)$, $P_0f(D)P_0-P_0f(\bar D)P_0 \in C^*(\tilde W\subset
  \tilde W_\infty)^\Gamma$. 
\end{lemma}
\begin{proof}
  We know that $f(D)\in C^*(\tilde W_\infty)^\Gamma$ and $f(\bar D)\in
  C^*(\reals\times\boundary\tilde W)^\Gamma$ and that $P_0$ is in $D^*$. So it
  only remains to check the support condition, i.e.~to show that for each
  $\epsilon>0$ there is an $R>0$ such that $\norm{P_0f(D)P_R-P_0f(\bar
    D)P_R}<\epsilon$ (then also its adjoint $P_Rf(D)P_0-P_Rf(\bar D)P_0$ has norm
  $<\epsilon$). For the latter, we approximate $f$ in supremum norm by a
  function
  $g$ whose Fourier transform has compact support (say in $[-R,R]$), and
  therefore approximate $f(D)$ in operator norm by $g(D)$  (and $f(\bar D)$ by
  $g(\bar D)$). The usual Fourier inversion formula then implies, using unit
  propagation speed for the wave operators of $D$ and $\bar D$, that $P_0g(D)
  P_R=P_0g(\bar D)P_R$, and the statement follows.
\end{proof}

We now generalize Lemma \ref{lem:fD_loc} to $D+B$. Note that we don't have
unit propagation speed available. 
%We expect that it should be
%possible to develop the corresponding
%(weakened) properties of the wave operator $\exp(it (D+C))$, however, we have
%not carried this out in detail.  \footnote{From Paolo: actually we do discuss this later}
Instead, we give a proof which
uses a comparison between $f(D)$ and $f(D+B)$.

\begin{lemma}\label{lem:fDC_loc}
  For $f\in C_0(\reals)$ and $B$ as in \eqref{def-of-B} 
  we have that $P_0 f(D+B)P_0-P_0f(\bar D+\bar C)P_0 \in C^*(\tilde W\subset
  \tilde W_\infty)^\Gamma$. 
\end{lemma}
\begin{proof}
For the following simple argument, using the Neumann series,
  we are 
  grateful to the referee.
We have to show that for $f\in C_0(\reals)$ it holds that $\lim_{R\to\infty}
  \norm{P_Rf(D+B)P_0-P_Rf(\bar D+\bar C)P_0}\xrightarrow{R\to\infty} 0$; the
  corresponding statement then also holds for the adjoint, involving
  $P_0f(D+B) P_R$.
  
Recall that 
    if $X$ is a norm limit of bounded finite propagation
  operators then 
  $\lim_{\abs{R-S}\to\infty}\norm{(1-P_S)X P_R }=0$. 
   Therefore, if $A,\bar A,B,\bar B$ are norm limits of $\Gamma$-equivariant finite propagation
   operators and if $P_0AP_0-P_0\bar A P_0, P_RBP_0-P_R\bar BP_0\in C^*(\tilde
   W\subset \tilde W_\infty)^\Gamma$ in norm, then also
   \begin{multline}\label{eq:prod_lim}
     P_{2R}ABP_0 - P_{2R} \bar A\bar B P_0
 = \underbrace{(P_{2R}AP_R - P_{2R}\bar A P_R)}_{\xrightarrow{R\to\infty} 0}
     P_R B P_0\\
 + P_{2R} \bar A P_R\underbrace{(P_R B P_0-P_R\bar B
       P_0)}_{\xrightarrow{R\to\infty} 0} + P_{2R} A \underbrace{((1-P_R)
     B P_0)}_{\xrightarrow{R\to\infty} 0} - P_{2R} \bar A \underbrace{((1-P_R)
     \bar B P_0)}_{\xrightarrow{R\to\infty}0} \xrightarrow{R\to\infty} 0\\
\implies P_0ABP_0-P_0\bar A\bar BP_0 \in C^*(\tilde W\subset \tilde
W_\infty)^\Gamma. 
   \end{multline}

Consequently,  if 
the lemma is true for $f$ and $g$ then it is true for their
 product $fg$. 
  As $C^*$-algebra, $C_0(\reals)$ is for any $\lambda>0$ generated by the two
  functions 
  $x\mapsto (x+i\lambda)^{-1}$ and $x\mapsto
  (x-i\lambda)^{-1}$. This means that it suffices to
  treat the functions $x\mapsto (x+\lambda i)^{-1}$ and $x\mapsto (x-\lambda
  i)^{-1}$ for $\lambda>0$ large  and we
  concentrate for notational convenience on the first.

As $\lambda$ is sufficiently large, we can then use the Neumann series
  \begin{equation*}
    (D+B+\lambda i)^{-1} = (D+\lambda i)^{-1} \sum_{k=0}^\infty
    \left(-B(D+\lambda i)^{-1}\right)^k;\qquad (\bar D+\bar B+\lambda i)^{-1}
    = (\bar D+\lambda i)^{-1} \sum_{k=0}^\infty
    \left(-\bar B(\bar D+\lambda i)^{-1}\right)^k
  \end{equation*}
By Lemma \ref{lem:fD_loc}, $P_0(D+\lambda I)^{-1}P_0-P_0(\bar D+\lambda
i)^{-1} P_0\in C^*(\tilde W\subset \tilde W_\infty)^\Gamma$ and by
\eqref{def-of-B} $P_0BP_0-P_0\bar BP_0\in C^*(\tilde W\subset \tilde
W_\infty)^\Gamma$. We \eqref{eq:prod_lim} we can pass to products
and finite sums, showing that for each $N\in\naturals$ also
\begin{equation*}
  P_o\left((D+\lambda i)^{-1} \sum_{k=0}^N
    \left(-B(D+\lambda i)^{-1}\right)^k\right)P_0 -   P_o\left((\bar D+\lambda i)^{-1} \sum_{k=0}^N
    \left(-\bar B(\bar D+\lambda i)^{-1}\right)^k\right)P_0 \in C^*(\tilde W\subset \tilde
W_\infty)^\Gamma
\end{equation*}
  Because of the norm convergence of the Neumann series then also
  \begin{equation*}
    P_0(D+B+\lambda i)^-1P_0 -P_0 (\bar D+\bar B+\lambda i)^{-1}P_0 \in
    C^*(\tilde W\subset \tilde W_\infty)^\Gamma.
  \end{equation*}
\end{proof}

To finish the proof of Proposition \ref{prop:key} we only have to prove
the support condition for $e(D+B)$, with $e(x)=1-\chi(x)^2$ as above, in
particular $e\in C_0(\reals)$. It follows from Lemma \ref{lem:fDC_loc} that
$$P_0e(D+B)P_0-P_0 e(\bar D+\bar C)P_0\in C^*(\tilde W\subset \tilde
W_\infty)^\Gamma,$$
 and by our choice of $\chi$ we have $e(\bar D+\bar
C)=0$. Finally, by finite propagation and the definition of $P_0$ also
$P_0e(D+B)(1-P_0)$, $(1-P_0)e(D+B)(1-P_0)$, $(1-P_0)e(D+B)P_0\in C^*(\tilde
W\subset \tilde W_\infty)^\Gamma$.   
   \end{proof}

% ===================  

\begin{definition}
  These considerations imply that there 
  %OLD COMMENT
%  \textcolor{red}{Thomas: wouldn't it be
  %  better to just give the class (as boundary element)} 
  is a well defined coarse
  relative index class 
\begin{equation}\label{rel-class}  
  \Ind^{{\rm rel}}(D_\infty + C_\infty) \in
  K_0(\relC{\widetilde{W}_\infty}{\widetilde{W}}^\Gamma).
  \end{equation}
  It is obtained by the standard construction in coarse index theory: with an 
  isometry $U\colon L^2 (\widetilde{W}_\infty, \widetilde{E}^-)
\to L^2 (\widetilde{W}_\infty, \widetilde{E}^+)$ which is $\Gamma$-equivariant
and which covers the identity in the  
$D^*$-sense (see \cite[Section 1]{PS-Stolz}) the operator $U\chi(D)_+$ is
invertible in
$D^*(\widetilde{W}_\infty)^\Gamma/\relC{\widetilde{W}_\infty}{\widetilde{W}}^\Gamma$
and therefore defines an element in $K_1$ of this $C^*$-algebra. $\Ind^{\rm
  rel}(D_\infty+C_\infty)$ is its image under the boundary map of the
associated long exact sequence in $K$-theory.
    %OLD COMMENT
%    \footnote{\textcolor{red}{Thomas:
  %    you used $K_n$, did you want to venture discussing both even and odd
     % case?}} 
We are also interested in the associated \emph{coarse index class}
\begin{equation}\label{coarse-class}
  \coarseind(D,C):=c_*^{-1}\Ind^{{\rm rel}}(D_\infty + C_\infty) \in K_0
  (C^*(\widetilde{W})^\Gamma)
\end{equation}
Here we use the canonical inclusion $c\colon C^*(W)^\Gamma\to
\relC{W_\infty}{W}^\Gamma$ which induces an isomorphism in K-theory,
see \cite[Lemma 1.9]{PS-Stolz}.\end{definition}

Note that the right hand side of  \eqref{coarse-class} is just a notation; we have not really defined an operator
$C$.

\begin{proposition}\label{prop:stab-index-class}
  If $B_\infty$ is the norm limit of $\Gamma$-equivariant bounded finite propagation operators,
  $C_\boundary$ is a perturbation as above and we assume that
  $P_0 B_\infty P_0 - P_0 C_\infty P_0 \in
  \relC{\tilde W}{\tilde W_\infty}^\Gamma$ then $\Ind^{\rm rel}(D_\infty+B_\infty) =
  \Ind^{\rm rel}(D_\infty+C_\infty)$.

  If $[0,1]\ni t\mapsto C^t_{\partial}$ is a continuous family of
    perturbations such 
    that $D_\partial + C_\partial^t$ is invertible for all $t$, then $\Ind^{\rm
      rel}(D_\infty+C_\infty^t)$ is independent of $t$.
      
      Finally, if $\widetilde{M}\to M$ is a $\Gamma$-Galois covering of a closed
      compact manifold $M$ and if $[0,1]\ni t\mapsto G_t$ is a continuous family
      of trivializing perturbations for $D$, as in Definition \ref{def:rho-perturbed}, then 
      $\rho (D+G_t)$ and 
      $\rho_\Gamma (D+G_t)$
      are independent of $t$.
      
\end{proposition}
\begin{proof}
 For the first two statements
  write either $A_t:=t B_\infty+(1-t)C_\infty$ or $A_t=C_\infty^t$. 
  Then $[0,1]\ni
  t\mapsto D_\infty+A_t$ is norm resolvent continuous, as
  $(D_\infty+A_t+i)^{-1}-(D_\infty+A_s+i)^{-1} = (D_\infty
  +A_t+i)^{-1}(A_s-A_t)(D_\infty+A_s+i)^{-1}$ and
  $\norm{(D_\infty+A_t+i)^{-1}}\le 1$. 
  
  As the $C^*$-algebra $C[-\infty,\infty]$ of bounded continuous functions on $\reals$ with limits at $\pm\infty$ is generated by $(x+i)^{-1}$, $(x-i)^{-1}$, $1$ and $\phi(x)=x/\sqrt{1+x^2}$ and the first two are taken care of by norm resolvent continuity, it remains to show that $t\mapsto \phi(D_\infty+A_t)$ is norm continuous to conclude that $t\mapsto \chi(D_\infty+A_t)$ is norm continuous for arbitrary $\chi\in C[-\infty,\infty]$.

 To this end we write 
 $\| \phi (D_\infty + A_t)- \phi (D_\infty + A_s)\|$ as 
 $\| \phi (G)- \phi (G + S_{s,t})\|$, with $G= D_\infty + A_t$ and $S_{s,t}= A_s - A_t$.
 %It suffices to prove the result for $\chi (x)= x (1+x^2)^{-1/2}$. 
 We employ the following identity, see \cite[Proposition
5.9]{higson-roeI} (this identity plays a fundamental role in the proof of 
Proposition \ref{prop:hr-controlled-sharp}). 
\begin{equation*}
\begin{split}
\phi (G)-\phi (G+S_{s,t})= & \frac{1}{\pi}\int_1^{+\infty}
\frac{u}{\sqrt{u^2-1}}R(u) S_{s,t} R_{s,t} (u) du\\
&+  \frac{1}{\pi}\int_1^{+\infty} \frac{u}{\sqrt{u^2-1}}
R(-u) S_{s,t} R_{s,t} (-u) du
\end{split}
\end{equation*}
where  $R(u)= (G+iu)^{-1}$ and $R_{s,t} (u)= (G+S_{s,t} + iu)^{-1}$.
Now, the norm of any of the two resolvent is bounded by $1/|u|$;
thus 
$$\| \phi (G)-\phi (G+S_{s,t}) \|\leq \frac{2}{\pi} \| A_s - A_t\| \left( \int_1^\infty
\frac{u}{\sqrt{u^2-1}} \frac{1}{u^2} du \right) $$
This establishes the required continuity. 
In particular, for any chopping function $\chi$,  $t\mapsto
  \chi(D_\infty+A_t)$ is a {\it norm continuous} family of involutions in
  $D^*(\tilde W_\infty)^\Gamma/\relC{\tilde W}{\tilde W_\infty}^\Gamma$. 
 By
  homotopy invariance of K-theory classes, the first two assertions
  follows.
  
   A similar proof applies to the last statement once we observe that we can use the same chopping function taking values $\pm 1$ on the spectrum of all (invertible!) operators in the compact family $D_\infty+A_t$, $t\in [0,1]$. This follows from the norm continuity of $t\mapsto (D_\infty+A_t)^{-1}$, using $$(D_\infty+A_s)^{-1}= (D_{\infty}+A_t)^{-1}(1+ (A_s-A_t)(D_\infty+A_t)^{-1})^{-1}$$ and a Neumann series argument for the last term.
   
    \end{proof}

%===============================
%  
%  \begin{proposition}
%Let $(W,g)$ be a compact orientable Riemannian manifold with boundary. Let $E$ be a complex hermitian vector bundle
%over $W$ with a Clifford module structure. Let $\mathcal{F}\to W$ be a bundle of finitely
%generated projective $A$-modules, endowed with a flat connection. 
%Then for $\epsilon >0$,  we have 
%\begin{equation*}
%%\begin{split}
%P\in \Psi^{-\infty,\epsilon,\epsilon}_{A,b}  (W,E\otimes \mathcal{F}), \;I(P)=0\Rightarrow P\in C^* (W\subset  W_\infty,A)
% %\Psi^{0} (X,E\otimes \mathcal{V})\subset D^*(X,A)
% %\end{split}
%\end{equation*}
%\end{proposition}

\subsubsection{Compatibility of index classes}\label{subsub:compatibility}

We now prove the compatibility of this relative index class in coarse geometry
with the index class in the Mishchenko-Fomenko framework.\\
Let $D_\infty$ be a Dirac type
operator on $\widetilde{W}_\infty$ and let 
$\D_\infty$ denote the associated operator on $W_\infty$
with coefficients in the Mishchenko bundle $\mathcal{F}_{{\rm M}}$ .
 Let $\mathcal{C}_{\pa}$ be a  trivializing perturbation for  $\D_\pa$
 and let $\D_\infty+ \C_\infty$ the associated operator on $W_\infty$, obtained
 by applying the same construction explained above for the covering
 $\widetilde{W}$, see Definition \ref{def:global-perturbation}.
 {We concentrate directly on the case in which $\mathcal{C}_{\pa}$
 is smoothing (as far as compatibility is concerned, this is the case we are interested in).}
% As already explained , this operator can be either defined in the 
%$b$-pseudodifferential calculus
%or, alternatively, it can defined by estending $\mathcal{C}_{\pa}$
% trivially on a cylinder and then grafting this operator on the manifold
%with cylindrical end through a cut-off function.
By \cite[Theorem 6.2]{LPGAFA} and  \cite[Theorem 10.1]{LLP}, $(\D_\infty+ \C_\infty)^+$ is invertible modulo 
$\KK (\mathcal{E}_{{\rm M}})$. Further extensions of these 
results were subsequently given by Wahl in \cite{wahl-aps}.
%, with $\mathcal{E}_{{\rm M}}$
%denoting, as before,  the Mishchenko
%$C^*_r \Gamma$-Hilbert
%module of $L^2$ sections of $E\otimes   \mathcal{F}_{{\rm M}}$. 
Consequently, we get an index class
$$\Ind^b_{{\rm MF}} (\D_\infty+\C_\infty)\in K_{n+1} (\KK (\mathcal{E}_{{\rm M}}))=K_{n+1} (C^*_r \Gamma)$$

\begin{remark}
We can apply this construction to the manifold with boundary 
appearing in the definition of a cycle of $L_{n+1} (\ZZ\Gamma)$,
choosing as a Dirac operator $\D$ the signature operator and as a 
trivializing perturbation for the boundary operator $\D_{\pa}$ the operator
$\mathcal{C}_f$, the one induced by an homotopy equivalence $f$.
 We obtain an index class   
$\Ind^b_{{\rm MF}} (\D_\infty+ \C_{f,\infty})\in K_{n+1} (C^*_r \Gamma)$;
this is precisely the index class we have considered in Section \ref{sub:L-map}.
\end{remark}

%\begin{remark}
%Of course we could have chosen the Hilsum-Skandalis perturbation $\mathcal{A}_f$; 
%we would have obtained again a bounded perturbation $\mathcal{A}_\infty$ with 
%the property that $(\D_\infty+ \mathcal{A}_\infty)^+$ 
%is invertible modulo 
%$\KK (\mathcal{E}_{{\rm M}})$. One can prove, see 
%\cite{Wahl_higher_rho}, proof of Theorem 8.4,
%that
%$\Ind^b_{{\rm MF}} (\D_\infty+ \C_\infty)=\Ind^b_{{\rm MF}} (\D_\infty+ \mathcal{A}_\infty)$.
%\end{remark}

We have already observed in \cite{PS-Stolz} that there is 
an isomorphism $\KK (\mathcal{E}_{{\rm M}})\cong C^* (\widetilde{W}\subset
    \widetilde{W}_\infty)^\Gamma$. One can prove, as in \cite[Section
    2]{PS-Stolz}, the following proposition.
    
    \begin{proposition}\label{b=coarse}
    Under the canonical isomorphism $$K_{n+1} (\KK (\mathcal{E}_{{\rm M}}))
    \iso  K_{n+1} (C^* (\widetilde{W}\subset
    \widetilde{W}_\infty)^\Gamma)$$ one has the equality
    $$\Ind^b_{{\rm MF}} (\D_\infty+ \C_\infty)=\Ind^{{\rm rel}}(D_\infty +
    C_\infty). $$
    \end{proposition}

% \begin{remark}\label{remark:b=aps}
% There is also an APS-description of the index class defined in
% the previous proposition. Let $W$ be even dimensional.
% Consider again 
% the projection $\mathcal{P}=\chi_{[0,\infty)} (\D_{\pa}+\C)$
% %$$\mathcal{P}:= \frac{1}{2}\left( \frac{\D_{\pa}+\C}{|\D_{\pa}+\C |} + 1 \right)$$
% which was called a {\it spectral section} in \cite{LPGAFA}
% \cite{LP03} (following \cite{MPI}). Then one can define an APS boundary value problem
% and an index class $\Ind_{{\rm APS}} (\D,\mathcal{P})\in K_{\dim W} (C^*_r \Gamma)$, see 
% \cite{WuI}.
% By \cite[Theorem 5]{LP03}, $\Ind_{{\rm APS}} (\D,\mathcal{P})=\Ind^b_{{\rm MF}}
% (\D_\infty+ \C_\infty).$

% \end{remark}

\begin{remark}\label{remark:relative-formula}
The index class $\Ind^b_{{\rm MF}} (\D_\infty+ \C_\infty)$ heavily depends 
on the choice of trivializing perturbation $\C_{\pa }$. (Consequently, the same is true for
$\Ind^{{\rm rel}}(D_\infty +
    C_\infty)$.)
 Indeed, it is proved in
\cite[Theorem 6]{LP03}, inspired by \cite{MPI}, that if $\C_{\pa }^\prime$ is a different
perturbation then
\begin{equation}\label{relative-index-formula}
%\begin{split}
\Ind^b_{{\rm MF}} (\D_\infty+ \C_\infty)- \Ind^b_{{\rm MF}} (\D_\infty+ \C_\infty^\prime)
%& \equiv \Ind_{{\rm APS}} (\D,\mathcal{P}) - \Ind_{{\rm APS}} (\D,\mathcal{P}^\prime)\\ =
=
[\mathcal{P}^\prime- \mathcal{P}]\;\;\text{ in }\;\;K_{\dim W} (C^*_r \Gamma)
%\end{split}
\end{equation}
with $[\mathcal{P}^\prime- \mathcal{P}]$ the difference class of the two
projections $\mathcal{P}=\chi_{[0,\infty)} (\D_{\pa}+\C)$ and $\mathcal{P}^\prime=
\chi_{[0,\infty)} (\D_{\pa}+\C^\prime)$.

\end{remark}

\section{Delocalized APS-index theorem for perturbed operators}

Let $(\widetilde{W},g_{\widetilde{W}})$ be an even dimensional Riemannian orientable manifold with boundary
$\boundary \widetilde{W}$. Assume that $\Gamma$ acts freely, cocompactly and
isometrically on $\widetilde{W}$. {We consider the 
associated manifold with cylindrical ends
$\widetilde{W}_\infty$.
Let $C_{\pa}$ be a trivializing 
perturbation as before (i.e.~the norm limit of $\Gamma$-equivariant bounded self-adjoint finite propagation operators such that $D_{\pa}+C_{\pa}$ is $L^2$-invertible) and let 
 $C_\infty$ be the associated
perturbation on  $\widetilde{W}_\infty$,
as in Section \ref{subsub:cylindrical-classes}.}
Recall the  coarse index class
$ \coarseind(D,C) \in K_0
  (C^*(\widetilde{W})^\Gamma)$ defined in  \eqref{coarse-class}. Our main tool
  in this paper will be the following ``delocalized APS-index theorem for
  perturbed operators''.
\begin{theorem}
\label{theo:k-theory-deloc} 
The following equality holds
\begin{equation}\label{k-theory-deloc}
\iota_* ( \coarseind(D,C) )= j_*(\rho( D_{\pa} + C_{\pa})) \quad\text{in}\quad K_{0}
(D^*(\widetilde{W})^\Gamma).
\end{equation}
Here, we use $j\colon D^*(\boundary \widetilde{W})^\Gamma\to
D^*(\widetilde{W})^\Gamma$ induced by the inclusion 
$\boundary \widetilde{W}\to \widetilde{W}$ and $\iota\colon C^*(\widetilde{W})^\Gamma\to D^*(\widetilde{W})^\Gamma$ the inclusion.
\end{theorem}

\begin{corollary}\label{corollary-main}
  By functoriality, using the canonical $\Gamma$-map $u\colon \widetilde{W}\to E\Gamma$
  we have 
  \begin{equation*}
   \iota_* u_*(\coarseind (D,C))  =\rho_\Gamma( D_{\pa} + C_{\pa})
     \quad\text{in}\quad K_0 (D^*_\Gamma). 
 \end{equation*}
 If we define $\coarseind_\Gamma (D,C):=u_* (\coarseind (D,C)) \quad\text{in}\quad
 K_0(C^*_\Gamma)$
 then the last equation reads
  \begin{equation}\label{main-index-eq}
   \iota_* (\coarseind_\Gamma (D,C))  =\rho_\Gamma(D_{\pa} + C_{\pa} )
   \quad\text{in}\quad K_0(D^*_\Gamma).
 \end{equation}
\end{corollary}
 
\noindent 
We prove these results in Section \ref{sec:proof}.

\begin{remark}
%We comment here about  a small generalization of  theorem \ref{theo:k-theory-deloc}.
%Consider the Wiener-Hopf short exact sequence introduced 
%in \cite{mor-pia}. In this particular case this sequence can be written as
%\begin{equation}
%0\rightarrow \relC{\widetilde{W}_\infty}{\widetilde{W}}^\Gamma
%\rightarrow A^* (\widetilde{W}_\infty)^\Gamma
%\xrightarrow{\pi} B^* (\RR\times \partial\widetilde{W})^{\RR\times\Gamma}
%\rightarrow 0
%\end{equation}
Consider  
$B_\infty$,   a bounded self-adjoint $\Gamma$-equivariant odd
 operator on $W_\infty$ satisfying  \eqref{def-of-B}, i.e.
$B_\infty\;\text{{is a norm limit of  $\Gamma$-equivariant finite propagation operators}}$
and  $ 
  P_0B_\infty P_0 - P_0 C_\infty P_0\in \relC{\tilde W_\infty}{\tilde
 W}^\Gamma$.
{By Proposition \ref{prop:key}}, we have the % relative coarse index
% classes
%  $\Ind^{{\rm rel}}(D_\infty + B_\infty)$ 
%  and, thus, a
 coarse index class
 $ \coarseind(D,B) := c_*^{-1}\Ind^{{\rm rel}}(D_\infty + B_\infty) \in K_{n+1}
  (C^*(\widetilde{W})^\Gamma)$.
  Since  by Proposition \ref{prop:stab-index-class}
  $\Ind(D, B)=\Ind(D, C)$,
  we also have
  \begin{equation}\label{main-index-eq-pert}
  \iota_* ( \coarseind_\Gamma (D,B) )= \rho_\Gamma(D_{\pa} + C_{\pa} ).
  \end{equation}
\end{remark}

The above remark  brings us immediately to a version
of bordism invariance for $\rho$-classes:
\begin{corollary}\label{corol:part_index_special}
  Let $\widetilde{M}_1$ and $\widetilde{M}_2$ be two free cocompact 
$\Gamma$-manifolds of dimension $n$. Let $D_j$ be a
$\Gamma$-equivariant 
Dirac type
operator on $\widetilde M_j$, $C_j\in C^*(\widetilde{M}_j)^\Gamma$ a trivializing
perturbations ($j=1,2$).
Assume that $D_1$ and $D_2$ are bordant,
 i.e.~there exists a
   manifold $\widetilde{W}$ with free cocompact $\Gamma$-action
  such 
  that $\boundary \widetilde{W}=\widetilde{M}_1\disjointunion \widetilde{M}_2$
  and $D$ a $\Gamma$-equivariant Dirac type operator on $W$ with product
  structure near the boundary (in the sense of Dirac type operators) and with
  boundary operator $D_1\disjointunion - D_2$. Moreover, let $B_\infty$ be a
  global perturbation
  %\in \mathfrak{M} (C^* (\widetilde{W}_\infty)^\Gamma)$ 
    which is a norm limit of finite propagation operators and such that 
$  P_0 B_\infty P_0  -P_0 C^{\disjointunion}_\infty P_0\in \relC{\tilde W_\infty}{\tilde
 W}^\Gamma$, where 
 $C_\infty^{\disjointunion}:= C_{1,\infty}\disjointunion (- C_{2,\infty})$.
%to $(-\infty,0]\times \pa \widetilde{M}_1 \disjointunion (-\infty,0]\times \pa \widetilde{M}_2$
%and with the additional property that 

Assume in addition that $\coarseind_\Gamma (D,B) =0$ in 
 $ K_{n+1} (C^*_\Gamma)$. Note that this is satisfied e.g.~if $B_\infty$ is a
 \emph{trivializing} perturbation of $D$, i.e.~if $D+B_\infty$ is invertible.
Then
  \begin{equation*}
    \rho_\Gamma (D_1 + C_1) = \rho_\Gamma (D_2 + C_2) \in K_{n+1} (D^*_\Gamma).
  \end{equation*}
\end{corollary}
\begin{proof}
  The rho-class is additive for disjoint union. By assumption, $\coarseind_\Gamma (D,B) =0$. The assertion now follows
  directly from \eqref{main-index-eq-pert}.
\end{proof}

\begin{remark}\label{rem:symmetric}
We can now go back to Example \ref{example:LLP-rho} and 
explain why \begin{equation*}
\rho_\Gamma (D+\mathcal{S}_0)=\rho_\Gamma (D+\mathcal{S}_1) \quad \text{if } \mathcal{S}_0, 
\mathcal{S}_1 \text{ are two symmetric trivializing perturbations}.
\end{equation*}
Recall that we are considering $\widetilde{M}\to M$, a $\Gamma$-cover such that the differential-form Laplacian on the covering is $L^2$-invertible in degree $m$.
The above equality takes place in $K_{\dim M+1} (D^*_\Gamma)$.\\
We consider $[0,1]\times \widetilde{M}$ with associated manifold with
cylindrical ends.
By the delocalized APS-index theorem \ref{theo:k-theory-deloc} we have that
$$\rho_\Gamma (D+\mathcal{S}_0)-\rho_\Gamma (D+\mathcal{S}_1) = 
\iota_* (\Ind_\Gamma (D_\infty + \mathcal{S}^{0,1}_\infty))$$
with $\mathcal{S}^{0,1}_\infty$ the grafted perturbation associated to 
$\mathcal{S}_0$ on $\{0\}\times \widetilde{M}$ and $\mathcal{S}_1$ on 
$\{1\}\times \widetilde{M}$.
Using in a fundamental way the hypothesis that 
$\mathcal{S}^j$ are symmetric we have, see \cite[Proposition 4.4]{LPAGAG}, \cite{Wahl-K-theory}
$$
\Ind_\Gamma (D_\infty + \mathcal{S}^{0,1}_\infty) =
\Ind_\Gamma (D_\infty + \mathcal{S}^{0,0}_\infty)
\text{ in } K_* (C^*_r \Gamma)= K_* (C^*_\Gamma)
$$
Thus 
\begin{equation*}
\begin{split}
\rho_\Gamma (D+\mathcal{S}_0)-\rho_\Gamma (D+\mathcal{S}_1)&=
(\rho_\Gamma (D+\mathcal{S}_0)-\rho_\Gamma (D+\mathcal{S}_1))-
(\rho_\Gamma (D+\mathcal{S}_0)-\rho_\Gamma (D+\mathcal{S}_0))\\
&=\iota_* (\Ind_\Gamma (D_\infty + \mathcal{S}^{0,1}_\infty)) - \iota_*
(\Ind_\Gamma (D_\infty + \mathcal{S}^{0,0}_\infty)) =0.
\end{split}
\end{equation*}
\end{remark}

\begin{remark}
  As in \cite{PS-Stolz}, {it is plausible that 
  the methods which prove the delocalized APS-index
  theorem for perturbed operators, also yield a secondary partitioned manifold
  index theorem} \footnote{{We take this opportunity
  to point out that the secondary partitioned manifold index theorem 
  proved in \cite{PS-Stolz} can in fact be sharpened, with essentially the same proof,
   from an equality of {\it universal}
  rho-classes to an equality
  of rho-classes. We thank Vito Felice Zenobi and Rudolf Zeidler
  for alerting us of this sharpening.}}. {We leave the precise formulation and proof to future investigations.
  A potential application of such a theorem would be  the following}: 
  assume that $M_1,M_2$ are two
  oriented complete Riemannian manifolds with free proper isometric
  $\Gamma$-action and with $\Gamma$-invariant and $\Gamma$-cocompact
  separating 
  hypersurfaces $N_1\subset M_1$, $N_2\subset M_2$ with product
  neighborhoods. Let $f\colon M_1\to M_2$ be a $\Gamma$-equivariant homotopy
  equivalence which restricts to a homotopy equivalence $f|\colon N_1\to
  N_2$. Assume that $f$ coarsely preserves the signed distance to the
  hypersurfaces, 
  i.e.~that the two functions $d_1,d_2\circ f\colon M_1\to \reals$ with
  $d_1(x)= \epsilon_1(x)d(x,N_1)$ and $d_2(y)=\epsilon_2(y)d(y,N_2)$ are
  coarsely equivalent to each other. Here $\epsilon_1(x)=\pm 1$ depending
  whether $x$ lies in the positive or negative half of $M$.

  We call such a situation a ``partitioned manifold homotopy equivalence''. 

  In this situation, {possibly with further hypothesis 
  on $f$}, one  should obtain a partitioned manifold $\rho$-class
  $\rho^{{\rm pm}}(f)\in K_*(D^*_\Gamma)$. Similarly, we defined in
  \ref{example:rho-he} $\rho_\Gamma(f|_{N_1})\in K_*(D^*_\Gamma)$. The
  secondary partitioned manifold index theorem would claim  that these two classes
  coincide. 

 { As a potential application, assume that $f\colon N_1\to N_2$
 and $g:X_1 \to N_2$  are homotopy
  equivalences and that $\rho_\Gamma(f)\ne \rho_\Gamma(g)$
  so that $(N_1\xrightarrow{f} N_2)$ is not h-cobordant to $(X_1\xrightarrow{g} N_2)$.
  %This implies that $f$ can not be
  %deformed to a diffeomorphism. 
  Then, also between the  stabilizations $
  (N_1\times \reals\xrightarrow{f\times \id_\reals} N_2\times \reals)$ 
  and $
  (X_1\times \reals\xrightarrow{g\times \id_\reals} N_2\times \reals)$ 
 there is no Riemannian h-cobordism such that the inclusion maps of the
 boundaries are
 continuous coarse homotopy equivalences.  Otherwise, we expect
 %Indeed, if this was not the case then one
%should be able to prove that 
$\rho^{{\rm pm}} (f\times
  \id_\reals)= \rho^{{\rm pm}} (g\times
  \id_\reals)$; 
  however, by the secondary index theorem, this would contradict the hypothesis
  $\rho_\Gamma(f)\ne \rho_\Gamma(g)$.}
\end{remark}

\section{Mapping the surgery sequence to K-Theory}

\subsection{The structure set and the map $\rho_{\Gamma}$}
Let $V$ be an $n$-dimensional 
  smooth closed oriented manifold with fundamental group $\Gamma$, $n>4$.

\begin{definition}
The \emph{structure set}
  $\mathcal{S}^{h}(V)$ consists of equivalence classes $[M\xrightarrow{f} V]$,
  where $M$ is a smooth closed oriented manifold, and  $f$ is an
  homotopy equivalence.  $(M_1\xrightarrow{f_1} V)$ and
  $(M_2\xrightarrow{f_2}V)$ are equivalent if there is an h-cobordism $X$
  between $M_1$ and $M_2$ and a map $F\colon X\to V\times [0,1]$ (necessarily a homotopy
  equivalence) such that $F|_{M_1}=f_1$ and $F|_{M_2}=f_2$.
  \end{definition}

Let $V$ be as above. Let $u\colon V\to B\Gamma$ be
the classifying map for the universal cover of $V$ and let $\tilde{u}\colon \widetilde{V}\to E\Gamma$ be
a $\Gamma$-equivariant lift of $u$.
  Let  $M\xrightarrow{f}V$ be a cycle representing an element $[M\xrightarrow{f}V]$ in 
  $\mathcal{S}^{h}(V)$. As in example \ref{example:rho-he} we consider $Z= M\disjointunion (-V)$ with the obvious classifying map $u_Z
  \colon Z\to B\Gamma$ 
  induced  by $u$ and by $u\circ f$. We observe that $Z$ comes with a map $\phi$ to $V$, 
  $\phi:= f\disjointunion (-\Id_V)$,  which
  is covered by a $\Gamma$-equivariant map $\tilde{\phi}\colon \widetilde{Z}\to \widetilde{V}$.
 \begin{definition}\label{def:rho}
  We define a map 
   $\rho
  \colon \mathcal{S}^{h}(V)\mapsto 
   K_{n+1} (D^* (\widetilde{V})^\Gamma)$ as follows
 \begin{equation}\label{rho-map-0}
  \mathcal{S}^{h}(V)\ni  [M\xrightarrow{f}V]\to 
  \rho(f) \in K_{n+1}
  %\rho(f) \stackrel{\text{Def}}{=}\tilde{\phi}_* (\rho (D + C_f))\in K_{n+1}
    (D^* (\widetilde{V})^\Gamma)
    \end{equation}
  where we recall that $D$  is the signature operator
  on the covering $\widetilde{Z}:= u_Z^* E\Gamma$, $C_f$ is the smoothing
  trivializing perturbation constructed using the homotopy equivalence $f$
  and $\rho(f) := \tilde{\phi}_* (\rho (D + C_f))\in K_{n+1}
    (D^* (\widetilde{V})^\Gamma)$.
\end{definition}

\begin{definition}\label{def:rho-gamma}
We define 
  \begin{equation}\label{rho-map}
  \rho_\Gamma
  \colon \mathcal{S}^{h}(V)\to 
   K_{n+1} (D^*_\Gamma);  [M\xrightarrow{f}V]\mapsto \rho_\Gamma(f)
 % \rho_\Gamma(f)\stackrel{\text{Def}}{=}\rho_\Gamma (D + C_f)
    \end{equation}
    where we recall that $$\rho_\Gamma(f):= \rho_\Gamma (D + C_f):= (\tilde{u}_Z)_* (\rho (D + C_f))=
    \tilde{u}_* (\rho (f))\,.$$
  If necessary, we shall denote the right hand sides of 
  \eqref{rho-map-0} and 
\eqref{rho-map} as $\rho  [M\xrightarrow{f}V]$  and $\rho_\Gamma  [M\xrightarrow{f}V]$
respectively. 
 \end{definition}

From now on we assume our manifold $V$ to be odd dimensional.

\begin{remark}\label{remark:odd-versus-even}
 Notice
that we follow here different conventions with respect to \cite{higson-roeIII};
indeed we follow the conventions of \cite{PS1}, slightly modified  as in 
\cite{Wahl_higher_rho}. In particular, the operator we use on an odd
dimensional manifold is not the odd signature operator as in Atiyah, Patodi,
Singer (and the one used by topologists to define the signature K-homology
class, compare \cite{Rosenberg-Weinberger}). It is rather the direct sum of two
(unitarily equivalent) versions of this operator, and comes up as the boundary
operator of the signature operator on an even dimensional manifold.
 This results in a possible loss of two-torsion information. On the other
 hand, with this modification our diagrams commute without inverting $2$.

  However, 
 a factor of $2$ would with our conventions show up  in the same portion of
 the sequence for an even dimensional manifold $V$.

It would be intersting to improve the constructions of our invariants, in
particular the rho-invariant, so that they deal with the signature operator on
odd dimensional manifolds as of Atiyah-Patodi-Singer (potentially containing
additional $2$-torsion information). This involves understanding how the
Hilsum-Skandalis-Wahl perturbation can be made compatible with the splitting
of the operator alluded to above. It also involves establishing bordism
invariance of the rho-invariant to be obtained that way. This is a non-trivial
task, as shows \cite{Rosenberg-Weinberger}, where the corresponding problem of
bordism invariance of the K-homology class is solved.

Finally, we believe that our diagram can be extended to the left, provided
that one 
inverts $2$ (and multiplies all maps with the appropriate powers of $2$,
similar to \cite{higson-roeIII}), and in such a situation the powers of $2$
discussed above will inevitably appear to make all squares commutative.

\end{remark}

%Next we 
  %need to prove that these maps descend to $\mathcal{S}^{h}(V)$:
%  Notice
%  that the results of Wahl, here and in the proof of the next
%  Proposition, are stated in the Mishchenko-Fomenko
%framework; however because of our discussion in Section \ref{sec:indexdef}
%these results can be easily translated to the coarse index theory framework. 

\begin{proposition}\label{prop:well-rho}
The maps $$\mathcal{S}^{h}(V)\ni  [M\xrightarrow{f}V]\mapsto 
  \rho_\Gamma (f)\in K_{n+1} (D^*_\Gamma)\;\;\;\text{and}\;\;\;
  \mathcal{S}^{h}(V)\ni  [M\xrightarrow{f}V]\mapsto 
  \rho (f)\in K_{n+1} (D^* (\widetilde{V})^\Gamma)$$ are well defined.
\end{proposition}

\begin{proof}
This follows from the work \cite{Wahl_higher_rho} of Charlotte Wahl  and
 our delocalized APS index theorem.
 Let us first prove that
$\rho (f)\in K_0 (D^* (\widetilde{V})^\Gamma)$ is well 
defined.
We assume as before that  $[M_1 \xrightarrow{f_1} V]=[M_2 \xrightarrow{f_2} V]$
in $\mathcal{S}^h (V)$ and let
$F:X\to V\times [0,1]$ be the h-cobordism realizing this
equality;
we set again $W=X\disjointunion V\times [0,1]$.
Then there is a map $\Phi:W\to V$
given by $F\disjointunion \id_{V\times [0,1]}\colon W\to V\times [0,1]$, followed by the
projection $\pi_1$ onto the the first factor of $V\times [0,1]$; i.e.~$\Phi=
\pi_1 \circ (F\disjointunion \id_{V\times [0,1]})$. 
Let $\phi_1= f_1\disjointunion (-\Id_V): M_1 \disjointunion (-V)\to V$ and
similarly for $\phi_2$.
Set as before $Z_1:= M_1 \cup (-V)$ and $Z_2:=M_2 \cup (-V)$. We have observed that 
these manifolds come with classifying maps
$u\circ \phi_j$
into $B\Gamma$ and  we have denote by $\widetilde{Z}_j$ the corresponding $\Gamma$-covers.
As before, we consider the $\Gamma$-equivariant lifts $\tilde{\phi}_j\colon
\widetilde{Z}_j\to \widetilde{V}$.
Recall now that $\rho (f_1)\in K_0 (D^* (\widetilde{V})^\Gamma)$ is obtained 
by pushing forward through $(\tilde{\phi}_1)_* $ the rho-class of $\widetilde{Z}_1$,
denoted  $\rho (D_{\widetilde{Z}_1} + C_{f_1})\in  K_0 (D^* (\widetilde{Z}_1)^\Gamma)$. 
 Similarly $\rho (f_2)= (\tilde{\phi}_2)_* \rho (D_{\widetilde{Z}_2} + C_{f_2})$.
Clearly $\Phi$ restricted to $M_1 \disjointunion (-V)=:Z_1$ is $\phi_1$ and
similarly
for $\Phi$ restricted to $M_1 \disjointunion (-V)=:Z_2$. We can now apply
$\Phi_*$
to the delocalized APS index formula for $W$, which we write with self-explanatory
notation as
\begin{equation}\label{well-def-rho}
(\iota_{\widetilde{W}})_*  \Ind (D_{\widetilde{W}}, C_{1,2})= (j_1)_* \rho
(D_{\widetilde{Z}_1} + C_{f_1}) - (j_2)_* 
\rho ( D_{\widetilde{Z}_2} + C_{f_2} )\,,\end{equation}
with $j_1$ and $j_2$ the obvious inclusions. This equality takes place in 
$K_0 (D^* (\widetilde{W})^\Gamma)$.
{Now, 
Wahl proves in \cite[Theorem
8.4]{Wahl_higher_rho}, extending the work 
of Hilsum-Skandalis to manifolds with cylindrical ends, that the existence
of the map $F$ entering into the h-cobordism, i.e.~of a global homotopy equivalence, 
%defines a global perturbation $B_\infty$ of the signature operator $D_\infty$
%on $\widetilde{W}_\infty$. The perturbation $B_\infty$ extends the two  perturbations 
%defined by $f_1$ and $f_2$ on the cylindrical ends
%and, moreover, is such that the coarse index class of 
%$D_\infty +B_\infty$ is zero.  
%The result then follows immediately from the remark following 
%Corollary  \ref{corol:part_index_special}. 
implies the vanishing of the index class 
 obtained by grafting the two smoothing perturbations $C_{f_j}$ on the cylindrical ends (Wahl works in the Mishchenko-Fomenko framework
but since we know that this is equivalent to the coarse framework we can
directly state her results in the way we just have).
Since the latter is precisely the index class that was denoted $\Ind (D_{\widetilde{W}}, C_{1,2})$
in the left hand side of \eqref{well-def-rho},}
 we obtain at once that $$0=(j_1)_* \rho (D_{\widetilde{Z}_1} + C_{f_1}) -  
(j_2)_* 
\rho ( D_{\widetilde{Z}_2} + C_{f_2} )\;\;\text{in}\;\;K_0 (D^* (\widetilde{W})^\Gamma).$$
We now apply $\Phi_*$ to both sides and we obtain the following equality in
$K_0 (D^* (\widetilde{V})^\Gamma)$:
$$0= \Phi_* ((j_1)_* \rho (Z_1)) - \Phi_* ((j_2)_* \rho (Z_2))$$
However, by functoriality, the right hand side is exactly
$\rho(f_1)-\rho(f_2)$. Thus $0=\rho(f_1)-\rho(f_2)$ and we are done.

This argument also  proves that $\rho_\Gamma$
 is well defined; indeed it suffices to apply $\tilde{u}_* : K_0 (D^* (\widetilde{V})^\Gamma)
 \to K_0 (D^*_\Gamma)$ to the equality $0=\rho(f_1)-\rho(f_2)$ and recall that
 $\tilde{u}_* (\rho(f_j))=\rho_\Gamma (f_j)$. See the remark after \eqref{def-of-rho-Gamma(f)}.
 \end{proof}

\subsection{The set $\protect\mathcal{N} (V)$ and the map $\beta_\Gamma : \protect\mathcal{N} (V) \to K_n (B\protect \Gamma)$}\label{sub:beta}

%$\protect\mathcal{N} (V)$ 
%$\beta_\Gamma\colon\protect\mathcal{N} (V) \to K_n(B\protect\Gamma)$}
  %\otimes \ZZ [ \frac{1}{2} ]$}

Let $V$ be an $n$-dimensional 
  smooth closed oriented manifold with fundamental group $\Gamma$. We assume
 that $n>4$ and that it is odd.
Let $u\colon V\to B\Gamma$ be
the classifying map for the universal covering of $V$ and let $\widetilde{V}:= u^* E\Gamma$.

\begin{definition}
 $\mathcal{N}(V)$ is the set of normal maps. Its
  elements are pairs $[M\xrightarrow{f} V]$ where $M$ is again a closed smooth
  oriented manifold and  $f$ is a degree $1$ normal map.
Two such pairs are equivalent if there is a normal cobordism between them
  (for all this compare for example \cite[Section 5]{higson-roeIII} and the references 
  therein). 
  \end{definition}
  
  The map $\beta_\Gamma$ is the composition of 
  $\beta\colon \mathcal{N}(V)\to K_*(V)$ and the group homomorphism $u_* \colon K_*(V)\to K_* (B\Gamma)$ induced by the classifying map of $V$.
  %\otimes \ZZ [ \frac{1}{2} ]$ 
Here  $\beta$  maps 
  $[M\xrightarrow{f} V]$ in $\mathcal{N}(V)$ to 
  %$2^{-\lfloor  n/2 \rfloor} 
  $f_*[D_M]-[D_V] \in K_*(V)$,
%  \in K_*(V)\otimes \ZZ [ \frac{1}{2} ]$,
  where $[D_M]$, $[D_V]$ are the K-homology classes of the signature
  operators. 
  The map $\beta\colon \mathcal{N}(V)\to  K_*(V)$ already appears in the work of Higson and Roe where it is proved to
  be well defined.
  
 \subsection{Mapping the surgery  sequence to the Higson-Roe
   sequence}
 
 \begin{theorem}\label{theo:main}
 Let $V$ be an $n$-dimensional 
  smooth closed oriented manifold with fundamental group $\Gamma$. We assume that
  $n>4$ is odd. Then there  is a  commutative diagram with exact rows
  
   \begin{equation}\label{PS}
  \begin{CD}
  %@>>>
   L_{n+1}(\ZZ\Gamma) &\dashrightarrow& \mathcal{S}(V) @>>> \mathcal{N}(V) @>>> L_{n}(\ZZ\Gamma)\\
   %&& 
   @VV{\Ind_\Gamma}V  @VV{\rho}V @VV{\beta}V 
   @VV{\Ind_\Gamma}V \\
   %@>>> 
   K_{n+1} ( C^*_r\Gamma) @>>>  K_{n+1}(D^* (\widetilde{V})^\Gamma )
   @>>>K_{n} (V) @>>>  K_{n} ( C^*_r\Gamma)\\
    \end{CD}
    \end{equation}
By employing the classifying map $u\colon V\to B\Gamma$ for the universal cover $\tilde{V}$
of $V$ we also get  a commutative diagram mapping into  the
{\em universal} Higson-Roe  surgery sequence:
 \begin{equation}\label{PS-universal}
  \begin{CD}
 %@>>> 
 L_{n+1}(\ZZ\Gamma) &\dashrightarrow& \mathcal{S}(V) @>>> \mathcal{N}(V) @>>> L_{n}(\ZZ\Gamma)\\
   %&& 
   @VV{\Ind_\Gamma}V  @VV{\rho_\Gamma}V @VV{\beta_\Gamma}V 
   @VV{\Ind_\Gamma}V \\
    %@>>>  
    K_{n+1} ( C^*_r\Gamma)  @>{i_*}>>  K_{n+1}(D^*_\Gamma ) 
   @>>>K_{n} (B\Gamma) @>>>  K_{n} ( C^*_r\Gamma) \\
    \end{CD}
    \end{equation}

 \end{theorem}
 
 \begin{proof}
 Thanks to the work of Wahl (see Theorem \ref{theo:charlotte}) and Proposition \ref{prop:stab-index-class} we know that the index
 homomorphism $\Ind\colon L_* (\ZZ\Gamma)\to K_* (C^*_r \Gamma)$ is well
 defined\footnote{Wahl concentrates on the case $*$ even, but for this particular result her
 arguments apply, with a minimum  amount of work, to
  the case $*$ odd.}. By Proposition \ref{prop:well-rho}, $\rho$  is well defined
 and we know from \cite[Definition 5.2]{higson-roeIII} that $\beta$ is also
 well defined. We now proceed to
 prove the commutativity of the squares. 
 
 Recall first of all
   the meaning of exactness of the surgery sequence at $\mathcal{S}(V)$, see
  for example  \cite{higson-roeIII}. Let
  $[M\xrightarrow{f}V]\in 
\mathcal{S}(V)$ and let  $a\in L_{n+1} (\ZZ\Gamma)$.
  Then 
 $a$ can be  represented by $(W, F,M\times [0,1], u_M\colon M\to B\Gamma)$ with $F\colon W\to M\times [0,1]$ a normal map of degree one between manifolds with boundary and $u_M= u\circ f$.
 Here $W$ is a manifold with boundary with two boundary components
 $\pa_0 W$ and $\pa_1 W$, $\pa_0 F:= F |_{\pa_0 W} \colon \pa_0 W\to \{0\}\times M$ 
 is a diffeomorphism and $\pa_1 F:= F |_{\pa_1 W} \colon \pa_1 W\to \{1\}\times M$ 
is a homotopy equivalence. For later use we set $X:=
W\sqcup M\times [0,1]$ and we remark that  $X$ comes with a  natural map  $u_f: X\to V$
with $u_f:= f\circ (\pi_1( F\circ \id_{M\times [0,1]}))$ and $\pi_1$ the projection onto the first factor. 
The action of $ L_{n+1} (\ZZ\Gamma)$
on $\mathcal{S}(V)$ is as follows: if $a\in L_{n+1} (\ZZ\Gamma)$ and $[M\xrightarrow{f}V]\in 
\mathcal{S}(V)$ then 
$$a\cdot  [M\xrightarrow{f}V] := [\pa_1 W\xrightarrow{f\circ \pa_1 F} V]\,.$$  
Exactness of the surgery sequence at $\mathcal{S}(V)$ means the following:
two elements in the structure set belong to the same orbit under the action of
the group $L_{n+1} (\ZZ\Gamma)$ if and only if their images in 
$\mathcal{N}(V)$ coincide. By definition, in this situation where we deal with
group actions instead of homomorphisms, commutativity of the first square means
  $$\rho (a \cdot [M\xrightarrow{f}V])- \rho [M\xrightarrow{f}V]=
  i_* \Ind (a)\quad\text{in}\quad K_{n+1}(D^* (\widetilde{V})^\Gamma )$$
  i.e. that 
  $$\rho  [\pa_1 W\xrightarrow{f\circ \pa_1 F} V] - 
  \rho [M\xrightarrow{f}V]=  i_* \Ind (a)\,.$$
  The homomorphism $i_*$ appearing on the right hand side is, up to a canonical isomorphism, 
  the map induced by the inclusion $C^* (\widetilde{V})^\Gamma\to D^* (\widetilde{V})^\Gamma$
  whereas 
  the index class $\Ind (a)$ can be taken to be  the push-forward, from 
  $K_{n+1} (C^* (\widetilde{X})^\Gamma)$ to $K_{n+1} (C^* (\widetilde{V})^\Gamma)$,
  of the index class defined by the perturbed signature operator on
  $\widetilde{X}$:
  \begin{equation}\label{ind-a}
  \Ind (a)= (\tilde{u}_f)_* (\Ind (D_{\widetilde{X}} + C_{ \pa_0 F, \pa_1 F}))
  \end{equation}
  with $\tilde{u}_f$ the $\Gamma$-map covering $u_f\colon X\to V$.
    Now, the proof of  \cite[Proposition 7.1]{Wahl_higher_rho}, together with 
  Proposition \ref{prop:stab-index-class},   
 shows that we have the following identity in $K_{*+1}
 (D^*(\widetilde{V})^\Gamma )$: 
  \begin{equation}\label{rho-composition}
  \rho [L\xrightarrow{f\circ g} V] + \tilde{f}_* (\rho [M\xrightarrow{\Id_M} M])=
  \tilde{f}_* (\rho [L\xrightarrow{g} M]) + \rho [M\xrightarrow{f} V]\,.
  \end{equation}
 In order to prove Equation \eqref{rho-composition} we consider the manifold
 $Y:=L\sqcup (-M) \sqcup M \sqcup (-V) $ and the natural inclusions
  $$ Z_1:= L\sqcup (-M) \xrightarrow{j_1} Y,\quad \quad Z_2:= M\sqcup (-V) \xrightarrow{j_2} Y,\quad T_1:= L\sqcup (-M) \xrightarrow{i_1} Y\quad, \text{and}\quad T_2:= M\sqcup (-V) \xrightarrow{i_2} Y.$$
One gets the invertible operators
$D_{\widetilde{Z}_1} + C_g$ and $ D_{\widetilde{Z}_2} + C_f$
and the associated rho classes
$$\rho (D_{\widetilde{Z}_1} + C_g)\in K_{n+1} (D^* (\widetilde{Z}_1)^\Gamma)
\;;\quad \rho (D_{\widetilde{Z}_2} + C_f)\in K_{n+1} (D^* (\widetilde{Z}_2)^\Gamma)\;.$$
Similarly, we obtain the invertible operators
$D_{\widetilde{T}_1} + C_{f\circ g}$ and $D_{\widetilde{T}_2} + C_{\Id}$
and the rho classes
$$\rho (D_{\widetilde{T}_1} + C_{f\circ g})\in K_{n+1} (D^* (\widetilde{T}_1)^\Gamma)
\;;\quad \rho (D_{\widetilde{T}_2} + C_{\id})\in K_{n+1} (D^*
(\widetilde{T}_2)^\Gamma)\;. $$
We claim that
\begin{equation}\label{compo-pre}
(\tilde{j}_1)_* (\rho (D_{\widetilde{Z}_1} + C_g))+ (\tilde{j}_2)_* (\rho (D_{\widetilde{Z}_2} + C_f))
=
(\tilde{i}_1)_* (\rho (D_{\widetilde{T}_1} + C_{f\circ g})) + 
(\tilde{i}_2)_* ( \rho (D_{\widetilde{T}_2} + C_{\Id}))
\end{equation}
in the group  $K_{n+1} (D^* (\widetilde{Y})^\Gamma)$.
Indeed, by a simple functoriality argument,
both sides are defined by rho classes of suitable perturbations
of the signature operator on $\widetilde{Y}$; Wahl shows in the proof of  \cite[Proposition 7.1]{Wahl_higher_rho} 
that there is a continuous family of trivializing perturbations,
 $[0,1]\ni t \to G_t$, interpolating from the perturbation defining the left hand side
 to the perturbation defining the right hand side. Thanks to   
 Proposition \ref{prop:stab-index-class} we then obtain immediately \eqref{compo-pre}.
 We now consider the map $\Psi:Y\to V$ defined as 
 $$\Psi |_{L}:= f\circ g\,,\;\; \Psi |_{M}:= f\,,\;\; \Psi |_{-V}:= \Id_{-V}$$
 and we observe that
 $$\Psi\circ i_1=(f\circ g)\sqcup \Id_{(-V)}\,,\;\;\Psi\circ i_2=f\sqcup (-f)\,,\;\;\Psi\circ j_1= f\circ (g\sqcup \Id_{(-M)})\,.$$
 By applying  $\tilde{\Psi}_*$ to both sides of \eqref{compo-pre}
 we then get, by functoriality, that
 $$\tilde{f}_* (\rho (g)) + \rho (f) =
  \rho (f\circ g) + \tilde{f}_* (\rho (\Id_M))\;\;\text{in}\;\;K_{n+1} (D^* (\widetilde{V})^\Gamma)$$
  which is precisely \eqref{rho-composition}.
  Granted  formula \eqref{rho-composition} we obtain, as in 
  \cite[Theorem 9.1]{Wahl_higher_rho}:
   \begin{equation}\label{rho-composition-quater}
  \rho (a \cdot [M\xrightarrow{f}V])- \rho [M\xrightarrow{f}V]=
  \tilde{f}_* (\rho [{\pa_1 W}\xrightarrow{\pa_1 F} M])- \tilde{f}_* (\rho
  [\pa_0 W\xrightarrow{\pa_0 F} M]).
  \end{equation}
  %  \begin{equation}\label{rho-composition-bis}
%  \rho (a [M\xrightarrow{f}V])- \rho_\Gamma [M\xrightarrow{f}V]=
%  \rho_\Gamma [{\pa_1 W}\xrightarrow{\pa_1 F} M]- \rho_\Gamma
%  [\pa_0 W\xrightarrow{\pa_0 F} M].
%  \end{equation} 
  By applying the delocalized APS index theorem \ref{theo:k-theory-deloc} for
  perturbed operators and functoriality we
  recognize in the difference on the right hand side the element $ i_* \Ind (a)$,
  with $\Ind (a)$ as in \eqref{ind-a}.
  The commutativity of the first square is therefore established.
  
  Next we tackle the commutativity of the second square, where we recall that 
  the horizontal map $\mathcal{S}( V)\to \mathcal{N} (V)$ is the forgetful map. 
  Let $[M\xrightarrow{f}V]\in \mathcal{S}(V)$ and consider $Z=M\sqcup (-V)$
   with the usual maps $\phi\colon Z\to V$, $\phi:= f\sqcup \Id_{(-V)}$, and $u_Z\colon Z\to B\Gamma$, $u_Z:=u\circ \phi$. We let $C_f$ be
   the trivializing perturbation defined by $f$, so that $D_{\widetilde{Z}}+C_f$
   is invertible, with spectrum disjoint from the interval
   $[-2\epsilon,2\epsilon]$.
   Recall that for any closed compact manifold $Z$ of odd dimension,
 endowed with a classifying map $u_Z \colon Z\to B\Gamma$,
  $$K_1 (Z)=K_{0}(D^*Z/C^*Z)= 
  K_{0}(D^* (\widetilde{Z})^\Gamma / C^* (\widetilde{Z})^\Gamma )$$
  with $\widetilde{Z}=u_Z^* E\Gamma$; in this realization of the K-homology groups, 
  the K-homology signature class $[D_Z]\in 
    K_1 (Z)$ 
  is obtained by considering the projection
  $[1/2 (\chi (D_{\widetilde{Z}})+1)]$ in 
    $K_{0}(D^* (\widetilde{Z})^\Gamma / C^* (\widetilde{Z})^\Gamma )$ with $\chi$ a chopping function.
   The following continuous path of operators
$$[0,1]\ni t \rightarrow [\frac{1}{2} ( \chi (D_{\widetilde{Z}}+tC_f) + 1)]\in K_{0}(D^* (\widetilde{Z})^\Gamma / C^* (\widetilde{Z})^\Gamma )$$
shows that for a chopping function $\chi$ equal to $1$ on $[\epsilon,+\infty)$
and $-1$ on $(-\infty,\epsilon]$ we have the equality
\begin{equation}\label{suffices0}
[\frac{1}{2} (\frac{D_{\widetilde{Z}}+C_f}{
|D_{\widetilde{Z}}+C_f |} + 1)]= [1/2 (\chi (D_{\widetilde{Z}})+1)]
\text{ in } K_{0}(D^* (\widetilde{Z})^\Gamma / 
C^* (\widetilde{Z})^\Gamma ) .
\end{equation}
With $\phi= f\sqcup \Id_{(-V)}$ we now apply  $(\tilde{\phi})_*$ to both sides
of this equality. On the left hand side
we obtain the class $\rho [M\xrightarrow{f}V] \in K_{0}(D^*
(\widetilde{V})^\Gamma / 
C^* (\widetilde{V})^\Gamma )=K_1 (V)$ whereas on the right we obtain $\phi_* [D_Z]\in K_1 (V)$ which
is precisely $f_* [D_M]-[D_V]$.
The commutativity of the second square is established.

Finally, we tackle the third square. The image of an element $[M\xrightarrow{f} V]\in \mathcal{N}(V)$
in $L_{n} (\ZZ\Gamma)$ is simply the class associated 
to the  cycle represented by the closed odd dimensional
 manifold $Z=M\cup (-V)$ endowed with 
the classifying map $u_Z$. This element in $L_{n} (\ZZ\Gamma)$ is the {\em surgery obstruction}
associated to $[M\xrightarrow{f} V]\in \mathcal{N}(V)$ and  
one of the fundamental   
results in surgery theory states that  a  normal map $f\colon M\to V$
can be surgered to a homotopy equivalence if and only if its surgery obstruction 
in $L_{n} (\ZZ\Gamma)$ vanishes. This is the meaning of exactness
of the surgery sequence at $\mathcal{N} (V)$.
Going back to the commutativity of the third square
we observe that the map $\Ind$ sends this class, the surgery obstruction
associated to $[M\xrightarrow{f} V]$,
to the associated index class in $K_{n} (C^*_r \Gamma)$. 
By \cite[Theorem 2]{conspectus}, this class is equal to the  
Mishchenko symmetric signature of the covering $u_Z^* E\Gamma \to Z\equiv
M\cup (-V)$. This remark reduces the proof of the  
commutativity of the third square to the proof given by Higson and Roe in
\cite[Theorem 5.4]{higson-roeIII}.
\end{proof}

 \begin{remark}
%   As mentioned above, we strongly expect that the result holds also in the
%   other parity. As already expressed in the companion paper \cite{PS-Stolz},
%   we hope that a direct calculation along very similar lines, perhaps using
%   $Cl_n$-linear operators and suitable properties of those, should yield the
%   desired result. We expect that the treatment of unperturbed operators then
%   passes to perturbed
%   operators, exactly as presented in the paper at hand, and gives the desired
%   result for homotopy equivalences and the signature operator.
In \cite{Xie-Yu}, a slightly different approach to the
   delocalized  APS index theorem for spin Dirac operators is worked
   out. It uses the localization algebras invented by Yu in
   \cite{YuLocalization}, which allow to
   formulate and prove product  formulas for products between fundamental
   classes (primary invariants) and rho-classes (secondary invariants). 
%   It
%   seems likely that this approach also can be generalized to the perturbed
%   operators we have to use here. However, one does face technical
%   difficulties related to the non-locality of the perturbed operators and the
%   very definition of localization algebras based on locality of the operators
%   involved there.
 Presumably this approach  can be generalized to the perturbed
   operators we have to use here; however, it is clear that one will
   face technical
   difficulties related to the non-locality of the perturbed operators and the
   very definition of localization algebras (based on locality of the operators
   involved).

 \end{remark}

\begin{remark}
So far we have treated only the case in which the manifold $V$
appearing in the structure set $\mathcal{S}^h (V)$ is odd dimensional.
Similarly, in \cite{PS-Stolz} we have only treated the group $\Pos^{\spin}_n (X)$
for $n$ odd.  
The problem, both here and in \cite{PS-Stolz},
is centered around  the delocalized APS index theorem, a crucial tool which we
prove only for {\it even dimensional} manifolds with boundary.

However, there is a clear strategy to deduce the remaining case from the one we
have established, which is also suggested in \cite{Xie-Yu}. One should
develop a ``suspension homomorphism'', mapping in functorial and compatible
ways (for manifolds $W$ with and without boundary)
$K_*(C^*(W)^\Gamma)$ to $K_{*+1}(C^*(W\times \reals)^{\Gamma\times\integers})$, and
$K_*(D^*(W)^\Gamma)$ to $K_{*+1}(D^*(W\times \reals)^{\Gamma\times\integers})$.

These homomorphisms have to map the rho-classes (in the case of a manifold
without boundary) and index classes (in general) of the Dirac
operators on $W$ to those on $W\times \reals$, and they have to be injective.
It is then an easy exercise to deduce the delocalized APS  index theorem for one
parity from the one for the other parity.

This program has now been implemented by Vito Felice Zenobi in \cite{zenobi}, building 
on previous work of Paul Siegel, see \cite{Siegel-phd} and
\cite{Siegel}. Consequently, the delocalised
APS index theorem presented in this article is now established in every dimension.
Similarly, Zenobi shows that the delocalised APS index theorem proved in
\cite{PS-Stolz} is valid in all dimensions.
The latter is also a consequence of the treatment given by Xie and Yu in
\cite{Xie-Yu}.

%Indeed, much of the structure required  already exists or is being
%developed. The crucial
%step, i.e. the definition of the suspension homomorphism $K_*(D^*(W)^\Gamma)$
%to $K_{*+1}(D^*(W\times \reals)^{\Gamma\times\integers})$ for $W$ closed and
%its
%compatibility with rho-classes, is already contained in the thesis of 
%Paul Siegel \cite{Siegel-phd} (see also \cite{Siegel}); further results in
%this direction are worked out by Vito Zenobi
%in \cite{zenobi}.

Notice that a general treatment of the suspension homomorphism and the relevant
compatibilities and product formulas, using Yu's localization
algebras, is carried out by Rudolf Zeidler in \cite{Zeidler}, making more
explicit some of the constructions in \cite{Xie-Yu}. The methods in
\cite{Zeidler} also cover the case of the signature operator and its rho
invariants and are therefore another route to derive the delocalzed APS index
theorem for odd dimensional manifolds from the one for even dimensional manifolds.
\end{remark}

\section{Proof of the delocalized APS index theorem}\label{sec:proof}
 
 In this section we prove Theorem \ref{theo:k-theory-deloc}, the delocalized APS index theorem for perturbed operators.  Recall from \cite[Section 4]{PS-Stolz} that
  the corresponding result for the Dirac operator of a spin manifold with positive scalar
 curvature on the boundary is proved in two steps: first we show how to reduce
 the validity of the
 theorem on $\widetilde{W}_\infty$ to a version of the theorem on
 $\RR\times \pa \widetilde{W}$ (which we call the ``cylinder delocalized index
 theorem for perturbed operators'');
 next a detailed analysis, with explicit computations, is carried out on the
 cylinder in order to establish the cylinder delocalized index theorem. We will
 follow this strategy also here.
 
 In this subsection we make more precise
 the  notation given in the proof of Proposition \ref{prop:key}; thus we
 denote by $p$ the characteristic function of 
 $[0,\infty)\times \pa \widetilde{W}$ inside $W_\infty$ and we denote by $P$ the corresponding
 multiplication operator.  Finally, as in the proof of Proposition \ref{prop:key},
 we denote by $p_0$ the characteristic function of  
 $[0,\infty)\times \pa \widetilde{W}$ inside
 the full cylinder 
  $\RR\times \pa \widetilde{W}$ and by $P_0$ the corresponding multiplication operator.

  \subsection{Reduction to the cylinder}
Once we have at our disposal Proposition \ref{prop:key}, together with its
proof, it is elementary to check that
the reduction to the  cylinder proceeds exactly as in \cite{PS-Stolz}.
Thus
\begin{enumerate}
\item We establish that $ P_0 \,\chi (\overline{D} + \overline{C})_+ \, P_0$ is an involution in $D^* ([0,\infty)\times \pa \widetilde{W})^\Gamma/ {D^* ( \{0\}\times \pa\widetilde{W}}\subset
[0,\infty)\times \pa \widetilde{W})^\Gamma$, where we recall that on the cylinder the bundle on which the signature operator acts
can be identified with the direct sum of two copies of the same bundle, with the obvious grading. 
We obtain $\pa [P_0  \,
\chi (\overline{D}+ \overline{C})_+ \, P_0]
\in K_0 (D^* ( \pa \widetilde{W}\subset
[0,\infty)\times \pa \widetilde{W})^\Gamma )$.
\item Similarly, we have a class $\pa [P \chi (\overline{D} + \overline{C}) P]\in  
K_0 (D^* (\widetilde{W}\subset \widetilde{W}_\infty)^\Gamma)$,
%Recall that we
%identify the
%image of $P$ in $L^2(\reals\times\widetilde{W})$ with the corresponding
%subspaace of $L^2(\widetilde{W}_\infty)$. This way, we consider
where it is obvious how to consider 
$P \chi(\overline{D}+\overline{C})P $ {as an element of}
%on
%$L^2(\widetilde{W}_\infty)$, indeed as an element 
$D^*(\widetilde{W}_\infty)^\Gamma$. 
\item Next we establish that 
\begin{equation}\label{key-for-chopping}
\chi (D+C) - P \chi (\overline{D}+ \overline{C}) P\;\;\in\;\;
D^* (\widetilde{W}\subset \widetilde{W}_\infty)^\Gamma\,. 
\end{equation}
We shall prove \eqref{key-for-chopping} at the end of this subsection, just
below.
\item Using the latter information, we show that
if $\iota_*\colon K_* (C^* (\widetilde{W}\subset \widetilde{W}_\infty)^\Gamma)\to
K_* (D^* (\widetilde{W}\subset \widetilde{W}_\infty)^\Gamma) $ is induced by
the inclusion, then $\iota_* (\Ind^{{\rm rel}} (D + C))=  \pa [P \,\chi
(\overline{D}+ \overline{C})_+ \, P] \in K_0 (D^* (\widetilde{W}\subset
\widetilde{W}_\infty)^\Gamma)$.
\item  We then show that if $j_{+}$ is the homomorphism 
$ K_* (D^* (\pa\widetilde{W}\subset [0,\infty)\times \pa\widetilde{W})^\Gamma)\to
K_0 (D^* (\widetilde{W}\subset \widetilde{W}_\infty)^\Gamma)
$, then $j_+ ( \pa [P_0 \, \chi (\overline{D}+ \overline{C})_+ \, P_0])= \pa [P \,\chi (\overline{D}+ \overline{C})_+ \, P]$ 
in $ 
K_0 (D^* (\widetilde{W}\subset \widetilde{W}_\infty)^\Gamma)$.
\item Next we assume  the
  \begin{theorem}[Cylinder delocalized index theorem for perturbed operators]\label{theo:cyl} If
    $j_{\pa}$ is the isomorphism $K_* 
    (D^* (\pa\widetilde{W})^\Gamma) \to K_* (D^* (\pa\widetilde{W}\subset
    [0,\infty)\times \pa\widetilde{W})^\Gamma)$, then
    \begin{equation}\label{main-cyl}
\pa [P_0 \, \chi
    (\overline{D}+ \overline{C})_+ \, P_0]=j_{\pa} \rho (D_\pa + C_\pa) .
      \end{equation}
  \end{theorem}
  \item Finally we recall the following commutative diagram from \cite[Proposition 3.1]{PS-Stolz}:
  \begin{equation}\label{eq:defjs}
  \begin{CD}
    K_*(D^*(\boundary \widetilde{W})^\Gamma) @>{j_*}>> K_*(D^*(\widetilde{W})^\Gamma)\\
    @V{\iso}V{j_\boundary}V @V{\iso}V{c}V\\
    K_*(\relD{[0,\infty)\times\boundary \widetilde{W}}{\boundary \widetilde{W}}^\Gamma) @>{j_+}>>
    K_*(\relD{\widetilde{W}_\infty}{\widetilde{W}}^\Gamma)
  \end{CD}
\end{equation}
%with $j_+$ and $j_{\pa}$ induced by the natural inclusions.

\end{enumerate}
Granted these properties one gets
\begin{equation*}
\iota_* (\Ind^{{\rm rel}} (D+C))=  \pa [P \,\chi (\overline{D}+ \overline{C})_+ \, P] = j_+ ( \pa 
[P_0 \, \chi (\overline{D}+ \overline{C})_+ \, P_0]) = j_+ ( j_{\pa} \rho
(D_\pa + C)). 
\end{equation*}
Once we apply $c^{-1}$ to both sides we obtain precisely the equality in the
theorem, i.e.~$\iota_* (\Ind (D, C)) =  j_* (\rho (D_\pa + C_\pa))$.

\medskip
\noindent
{\it Proof of \eqref{key-for-chopping}.}
We go back to the notation adopted in the proof of Proposition  \ref{prop:key};
thus with a small abuse of notation we don't distinguish between 
$P_0$ and $P$. Recall that we
identify the
image of $P_0$ in $L^2(\reals\times\widetilde{W})$ with the corresponding
subspace of $L^2(\widetilde{W}_\infty)$ (in order to lighten the notation we do not
write the vector bundles in the $L^2$-spaces). This way, we consider
$P_0\chi(\overline{D}+\overline{C})P_0$ as an operator on
$L^2(\widetilde{W}_\infty)$.

First of all we remark that thanks to Proposition \ref{prop:key}
we only need to establish \eqref{key-for-chopping} for one specific
chopping function; we therefore choose $\chi(x)=x/\sqrt{1+x^2}$.
We know, see \cite[Lemma 4.12]{PS-Stolz}, that 
$\chi (D) - P_0 \chi (\overline{D}) P_0 \in
D^* (\widetilde{W}\subset \widetilde{W}_\infty)^\Gamma $. We also know from
\cite[Lemma 5.8]{higson-roeI} that $$\chi (\overline{D}) - \chi (\overline{D}
+ \overline{C}) \in C^* (\RR\times\pa\widetilde{W})^\Gamma \;;\quad 
\chi (D) - \chi (D+C) 
\in C^* (\widetilde{W}_\infty)^\Gamma\,.$$
Then we write
\begin{equation*}\chi (D+C) - P_0 \chi (\overline{D}+ \overline{C}) P_0=
(  \chi (D+C) - \chi (D)) + (\chi (D) - P_0 \chi (\overline{D}) P_0)
+ (P_0 \chi (\overline{D}) P_0 - P_0 \chi (\overline{D}+ \overline{C}) P_0).
\end{equation*}
As already remarked,  the second summand on the right hand side is
an element in $D^* (\widetilde{W}\subset \widetilde{W}_\infty)^\Gamma$. We will 
prove that the sum of the first and third term
on the right-hand side is in  $C^* (\widetilde{W}\subset
\widetilde{W}_\infty)^\Gamma$.
To this end we write an explicit expression for this sum, using that $\chi(x)=x/\sqrt{1+x^2}$. Indeed, {using once again} \cite[Lemma 5.8]{higson-roeI}, we can write
this sum as operator norm convergent integral
\begin{multline*}
\frac{1}{\pi}\int_1^\infty \frac{t}{\sqrt{t^2-1}} \left( 
(D+it)^{-1} C (D+C+it)^{-1} -
P_0 (\overline{D} + it)^{-1} \overline{C} (\overline{D}+ \overline{C}  + it)^{-1} P_0
\right) dt\\
+ \frac{1}{\pi}\int_1^\infty \frac{t}{\sqrt{t^2-1}} \left( 
(D-it)^{-1} C (D+C-it)^{-1} -
P_0 (\overline{D} - it)^{-1} \overline{C} (\overline{D}+ \overline{C}  - it)^{-1} P_0
\right) dt .
\end{multline*}
Therefore  it suffices to show that
the integrands are in $C^* (\widetilde{W}\subset
\widetilde{W}_\infty)^\Gamma$. 
Let us consider the first integral and the t-dependent operator there:
\begin{equation}\label{first-expression}
(D+it)^{-1} C (D+C+it)^{-1} -
P_0 (\overline{D} + it)^{-1} \overline{C} (\overline{D}+ \overline{C}  +
it)^{-1} P_0 .
\end{equation}
Writing $1=P_0 + (1-P_0)$ and reasoning as we did at the end of the proof 
of Proposition \ref{prop:key} (after the proof of Lemma \ref{lem:fDC_loc}), 
we see that \eqref{first-expression}
is equal, up to a term in $C^* (\widetilde{W}\subset
\widetilde{W}_\infty)^\Gamma$,
to 
\begin{multline*}
P_0 (D+it)^{-1} P_0 C P_0 (D+C+it)^{-1} P_0 -
P_0 (\overline{D} + it)^{-1} P_0 \overline{C}  P_0 (\overline{D}+ \overline{C}  + it)^{-1} P_0\\
=P_0 (D+it)^{-1} P_0 C P_0 \left( P_0 (D+C+it)^{-1} P_0 - 
P_0 (\overline{D}+ \overline{C}  + it)^{-1} P_0 \right) +\\
\left( P_0 (D+it)^{-1} P_0 - P_0 (\overline{D} + it)^{-1} P_0 \right)
P_0 C P_0 (\overline{D}+ \overline{C}  + it)^{-1} P_0 
\end{multline*}
since, by definition,
$P_0  \overline{C} P_0 = P_0 C P_0$.
We already know that $\left( P_0 (D+C+it)^{-1} P_0 - 
P_0 (\overline{D}+ \overline{C}  + it)^{-1} P_0 \right)$
and $\left( P_0 (D+it)^{-1} P_0 - P_0 (\overline{D} + it)^{-1} P_0 \right)$
are elements in 
$C^* (\widetilde{W}\subset \widetilde{W}_\infty)^\Gamma$.
Since the other factors are all norm limits of finite propagation operators, thus multipliers
of $C^* (\widetilde{W}\subset \widetilde{W}_\infty)^\Gamma$,
we see that \eqref{first-expression} is in $C^* (\widetilde{W}\subset \widetilde{W}_\infty)^\Gamma$. The proof of \eqref{key-for-chopping} is complete.

\subsection{Proof of the cylinder delocalized index theorem for perturbed operators}

We now show Theorem \ref{theo:cyl}, namely
that  if $j_{\pa}$ is the isomorphism $K_* (D^* (\pa \widetilde{W})^\Gamma)
\to K_* (D^* (\pa \widetilde{W}\subset \RR_{\geq}\times  \pa \widetilde{W})^\Gamma)$
induced by the inclusion, then
$\pa [P_0 \, \chi (\overline{D}+ \overline{C})_+ \, P_0]=j_{\pa} \rho (D_\pa + C_\pa)\quad\text{in}\quad 
K_* (D^* ( \pa\widetilde{W}\subset \RR_{\geq}\times
\pa\widetilde{W})^\Gamma)$, where we abbreviate $\RR_{\geq}:=[0,\infty)$.

\noindent
 We wish to point out that our  arguments, although somewhat lengthy, are
  elementary. It is certainly possible to envisage a proof
based on the Volterra expansion for the wave operator of a perturbed Dirac
operator. However,
doing this properly does require some non-trivial work and, in addition, it
would not 
generalize easily to Lipschitz manifolds. This is why we have followed the
route presented below.

% \noindent
% {FROM PAOLO: should we restate the fundamental equality
% $\pa [P_0 \, \chi (\overline{D}+ \overline{C})_+ \, P_0]=j_{\pa} \rho (D_\pa + C_\pa)$ in terms of the Mayer-Vietoris
% boundary map and in the non-bounding case, as we did in Stolz' paper ? This would read
% $$\delta_{{\rm MV}} (\rho (D_{\cyl} + C_{\cyl}))=
% \rho (D +C) \quad\text{in}\quad K_0 (D^* (\widetilde{M})^\Gamma)$$
% with $D+C$ the invertible Dirac operator on the Galois covering without boundary $\widetilde{M}$,
% $D_{\cyl}$ the corresponding Dirac operator on $\RR\times \widetilde{W}$ and $C_{\cyl}=\Id_{\RR}\otimes C$.\\
% What do you think ?
% }

\begin{notation} 
We set $\widetilde{M}:= \pa \widetilde{W}$.
We denote the Dirac operator on $\widetilde{M}$, the cross-section
of the cylinder, by $D$ and
the Dirac operator on $\RR\times \widetilde{M}$ by $\overline{D}$.
We denote a  trivializing perturbation for $D$ as $C$
and we denote the resulting perturbation on  $\RR\times \widetilde{M}$,
the one obtained by extending $C$ to be constant in the $\RR$-direction,
by $\overline{C}$. We use
the symbol $\L^2 (\widetilde{M})$ for the $L^2$-sections of the relevant
Clifford bundle. Similarly we employ $\L^2 (\RR\times \widetilde{M})$ and 
$\L^2 (\RR_{\geq} \times \widetilde{M})$.
\end{notation}

 \begin{proposition}\label{prop:v-covers}
 The bounded  linear operator 
\begin{equation}\label{V}
V\colon \L^2 (\widetilde{M})\to \L^2 ([0,\infty)\times \widetilde{M}),\qquad
s\mapsto Vs,\;(Vs)(t):= \sqrt{2|D+C|}e^{-t|D+C|}(s)\,.
\end{equation}
 covers in the $D^*$-sense the inclusion $i\colon M\hookrightarrow
 \RR_{\geq}\times M$, $m\to (0,m)$.
 \end{proposition}
 
 \begin{proof}
 We prove this in  Section \ref{sec:V_covers}. \end{proof}

 \begin{proposition}\label{prop:deformed-in-d}
 The operator 
  $\frac{|D+C| + \pa_t}{(D+C)-\pa_t}$ belongs to 
 $ D^* (\RR\times \widetilde{M})^\Gamma$.
 \end{proposition}
 \begin{proof}
We prove this in Section \ref{sec:belong_to_D}.
 \end{proof}
  
  Granted these two Propositions, we can proceed as in \cite[Section
  4.4]{PS-Stolz}. We have the following proposition.

  \begin{proposition}\label{prop:Q}
    Set $Q:= P_0 \frac{|D+C| -\pa_t}{(D+C)+\pa_t} P_0 \colon \L^2
    (\RR_{\geq}\times\widetilde{M})\to \L^2 ( 
    \RR_{\geq}\times \widetilde{M})$. Then
 \begin{equation}\label{first} (P_0 \frac{|D+C| +\pa_t}{(D+C)-\pa_t} P_0)\circ
   Q= \Id_{\L^2 ( \RR_{\geq}\times\widetilde{M})};\qquad Q\circ (P_0
   \frac{|D+C| +\pa_t}{(D+C)-\pa_t} 
   P_0 )= \Id_{\L^2 (\RR_{\geq}\times\widetilde{M})}
   - \Pi\end{equation}
 with $\Pi:= V \chi_{[0,\infty)} (D+C) V^*$.
\end{proposition}

 This implies that 
 \begin{equation}\label{eq:compind}
 \pa [P_0  \frac{|D+C| +\pa_t}{(D+C)-\pa_t} P_0 ] = [V \chi_{[0,\infty)} (D+C) V^*]\quad\text{in}\quad 
 K_0 (D^* (\widetilde{M}\subset \RR_{\geq}\times \widetilde{M})^\Gamma)\,.
\end{equation}

 The proof of  Proposition \ref{prop:Q} proceeds as in \cite[Section
 4.4]{PS-Stolz}, given that the arguments there are
 purely functional analytic and rest ultimately on the Browder-Garding
 decomposition for the self-adjoint operator $D+C$.
 Now, on the right hand side of Equation \eqref{eq:compind} we have, by
 definition,
  $j_{\pa} \rho (D_\pa + C_\pa)$
 whereas the left hand side is equal to 
 $\pa [P_0 \, \chi (\overline{D}+ \overline{C})_+ \, P_0]$. 
 This assertion, namely that
 \begin{equation}\label{eq:bd}
\pa [P_0  \frac{|D+C| +\pa_t}{(D+C)-\pa_t} P_0]= 
 \pa [P_0 \, \chi (\overline{D}+ \overline{C})_+ \, P_0]
\end{equation}
  is proved as in \cite{PS-Stolz}, using exactly  the deformation argument
  explained there, after the proof of Proposition 4.33.
 Thus, assuming 
 Proposition \ref{prop:v-covers} and 
 Proposition \ref{prop:deformed-in-d}, we have proved the cylinder delocalized
 index theorem \ref{theo:cyl}.

 \subsubsection{Proof of Proposition \ref{prop:v-covers}: the map $V$  covers
   the {inclusion} in the  $D^*$-sense}\label{sec:V_covers}

{We recall first of all that, by definition, 
$V\colon \L^2 (\widetilde{M})\to \L^2 ([0,\infty)\times \widetilde{M})$
%,\qquad
%s\mapsto Vs,\;(Vs)(t):= \sqrt{2|D+C|}e^{-t|D+C|}(s)\,.
%\end{equation}
 covers in the $D^*$-sense the inclusion $i\colon \widetilde{M}\hookrightarrow
 \RR_{\geq}\times \widetilde{M}$ if 
  $V$ is the norm-limit 
    of bounded $\Gamma$-equivariant maps $U$ satisfying the following two conditions:
\begin{itemize}    
\item there is an $R>0$ such that $\phi U \psi =0$ if 
$d({\rm Supp} (\phi),i({\rm Supp} (\psi)))>R$,
    for $\phi\in C_c ( \RR_{\geq}\times \widetilde{M})$ and $\psi\in C_c (\widetilde{M})$;
    \item  $\phi U \psi$
    is compact for each $\phi\in C_0 ( \RR_{\geq}\times \widetilde{M})$
    and $\psi\in C_c (\widetilde{M})$ such that ${\rm Supp} (\phi) \cap i ({\rm Supp} (\psi))=
    \emptyset$
    \end{itemize}}

\begin{notation}
 {In all this subsection we shall denote by $M$ the total space of a 
Riemannian  $\Gamma$-Galois covering with compact base $M/\Gamma$.
As before, we write $D$ for a Dirac type operator on $M$, acting on sections
of a Clifford module bundle.  We will in the notation ignore this bundle.}
\end{notation}

Let $E\colon
L^2(M)\to L^2(M)$ be a {bounded equivariant self-adjoint operator}, norm limit
of equivariant finite propagation operators and with the property that 
 $D+E$ is invertible. 

Recall that we set
$V\colon L^2(M)\to L^2( [0,\infty)\times M) = L^2([0,\infty), L^2(M)); (Vu)(t)
=\sqrt{2\abs{D+E}} e^{-t\abs{D+E}} u$.

\begin{lemma}\label{lem:opnorm}
  Let $H$ be a Hilbert space and $W\colon H\to L^2([0,\infty),H)$ be an
  operator of the form $W(u)(t) = f_t(A)u$ for a self-adjoint operator $A$ on
  $H$ and a measurable function $ [0,\infty)\to C_b(\reals); t\mapsto f_t$.

  Then $\norm{W}^2 \le \sup_{\lambda\in \spec(A)}
  \int_0^{\infty}\abs{f_t(\lambda)}^2 \;dt$.
\end{lemma}

\begin{proof}
  By definition, for $u\in H$ we have $\abs{Wu}^2 = \int_0^{\infty}
  \abs{f_t(A)u}_H^2 \;dt$. Use the Brower-Garding decomposition as direct
  integral 
  $H=\int H_\lambda \;d\mu(\lambda)$ according to the spectral decomposition
  of the self-adjoint operator $A$. Then $u=\int_{\spec(A)} u_\lambda d\mu(\lambda)$
  with $u_\lambda\in H_\lambda$ and $\abs{u}_H^2=\int
  \abs{u_\lambda}_{H_\lambda}^2\;d\mu(\lambda)$. Moreover, for a function
  $f(\lambda)$ we have $f(A)u =\int_{\spec(A)} f(\lambda)
  u_\lambda\;d\mu(\lambda)$. Then
  \begin{equation*}
    \begin{split}
      \abs{Wu}^2 & =\int_0^{\infty} \abs{f_t(A)u}^2 \;dt       = \int_0^{\infty} \int_{\spec(A)} \abs{f_t(\lambda) u_\lambda}^2_{H_{\lambda}}
       \;d\mu(\lambda)\;dt\\
     & \stackrel{\text{Fubini}}{=} \int_{\spec(A)} \int_0^{\infty}
       \abs{f_t(\lambda)}^2 \;dt
     \abs{u_\lambda}_{H_{\lambda}}^2 d\mu(\lambda) 
    \le \left( \sup_{\lambda\in \spec(A)} \int_0^{\infty} \abs{f_t(\lambda)}^2\;dt
 \right)\cdot \underbrace{\int_{\spec(A)}
   \abs{u_\lambda}^2_{H_\lambda}\;d\mu(\lambda)}_{=\abs{u}^2_H} .
    \end{split}
\end{equation*}
\end{proof}

\begin{lemma}\label{lem:awaysmall}
  The operator $P_R V\colon L^2(M)\to L^2([0,\infty) \times M)$
  converges to 
  $0$ in norm for $R\to \infty$, where $P_R$ is the orthogonal projection onto
  $L^2( [R,\infty)\times M)$ 
\end{lemma}
\begin{proof}
  By assumption, $D+E$ is invertible, i.e.~there is $\epsilon>0$ such that
  $(-\epsilon,\epsilon)\cap \spec(D+E)=\emptyset$. We can write $P_R\circ V
  (u)(t) =f^R_t(D+E) u$ with $f^R_t(\lambda)=
  \begin{cases}
    0; & t>R\\  \sqrt{2\abs{\lambda}} \exp(-t\abs{\lambda}); & t\le R
  \end{cases}$. For $\abs{\lambda}\ge \epsilon$, the
  $L^2$-norm of $t\mapsto f^R_t(\lambda)$ uniformly tends to $0$ as $R\to
  \infty$. By Lemma \ref{lem:opnorm}, $\norm{P_RV}\xrightarrow{R\to\infty} 0$.  
\end{proof}

We now treat $(1-P_R)V$. Define $V_D\colon L^2(M)\to
L^2([0,\infty)\times M)$ by $V_Du(t) = \sqrt{2\abs{D}}
e^{-t\abs{D}}u$ i.e.~exactly as $V$, but with $D+E$ replaced by $D$. We want
to prove the properties for $V$ by a comparison with $V_D$, where corresponding
properties have been established in \cite{PS-Stolz}. For the comparison
we use a resolvent trick, which requires some preparation.

  \begin{lemma}\label{lem:absapp}
    The function $\sqrt{\abs{\lambda}} e^{-\abs{\lambda}}$ can on $\reals$
    be approximated in supremum norm by $f(\lambda)$ which is a polynomial in
    $1/(\lambda^2+1)$ (without constant term, i.e.~vanishing at infinity) in
    such a way that even $e^{-\abs{\lambda}}$ is 
    approximated in $L^2(\reals)$ by $f(\lambda)/\sqrt{\abs{\lambda}}$.
  \end{lemma}
  \begin{proof}
Consider the subalgebra
   of $C_0 (\RR)$ generated by 
    $w(t):=1/(1+t^2)$ and
   $v(t):=t/(1+t^2)$ which is a  subalgebra that separates
   points. Notice the relation  
   \begin{equation}\label{relation}
   v^2 =w - w^2.
%   (\lambda/(1+\lambda^2))^2 =
%   (1+\lambda^2)^{-1}- (1+\lambda^2)^{-2}.
 \end{equation}
 Observe that the image of this subalgebra under the operator of
 differentiation contains for  $n\ge 1$ the function $w^n v$; indeed 
 $w^\prime= -2wv$, so that $(w^n)^\prime= -2n w^n v $ for all $n\geq 1$.
 
 Consider now the subalgebra of $C_0 (\RR)$ generated by $w$ and
 $wv$. This
 subalgebra also separates points; thus, by Stone-Weiestrass, every function
 in $C_0 (\RR)$ can be approximated by 
 a polynomial $P(w,wv)$ with $P(0,0)=0$ (elements in $C_0 (\RR)$ vanish at infinity).
 Because of \eqref{relation} we see that any continuous even function in $C_0 (\RR)$
 can be approximated by a polynomial in $w$ without constant term, whereas any
 continuous 
 odd function in $C_0 (\RR)$ can be approximated by a polynomial of the form
 $P(w)v$, with $P(0)=0$. 
   Notice, finally, that because of the remark following \eqref{relation}
   the antiderivative of a function of the form $P(w)v$, $P(0)=0$ belongs to the algebra generated by $w$ and $v$.   
%   the even functions in this ring
%   are polynomials in $(1+\lambda^2)^{-1}$. One checks easily that the ring is
%   preserved by differentiation $d/d\lambda$ and therefore also by
%   integration, provided the integrand has no constant term, i.e.~vanishes at
%   $\infty$. By the Stone-Weierstrass theorem, every odd continuous function
%   on $\reals$ which vanishes at $\infty$ can in supremum norm be approximated
%   by an odd function which is a polynomial in $\lambda/(1+\lambda^2)$ and
%   $(1+\lambda^2)^{-1}$.
\noindent
   Given $\epsilon>0$, we apply all this to $h'(t)$, with
   $h(t):=u(t)e^{-\abs{t}}(1+t^2)$ 
   and $u(t):=
   \begin{cases}
     \sqrt{\abs{t}}; & \abs{t}\ge \epsilon \\ \frac{7}{4}\epsilon^{-3/2} t^2-
       \frac{3}{4}\epsilon^{-7/2} t^4; & \abs{t}\le \epsilon
   \end{cases}$. 

\noindent
Note that $\abs{u(t)e^{-\abs{t}} -\sqrt{\abs{t}}e^{-\abs{t}}} \le
    4\sqrt{\epsilon}$. 

As $h$ is even and
   continuously differentiable,  $h'$ is odd and continuous. Because
   of the exponential decay of $e^{-\abs{t}}$, $h'(t)$ vanishes at $\infty$.
  Therefore, we can find an odd function $Q_\epsilon(t)=R(w(t))v(t)$, with $R$
  a polynomial without constant term, such that
  $\abs{h'(t)-Q_\epsilon(t)}<\epsilon$ for all $t\in \reals$.

  Then the even function $P_\epsilon(t):=\int_0^t Q_\epsilon(\tau)\;d\tau$ is a polynomial
without constant term in $w(t)=1/(1+t^2)$ such that $\abs{h(t)-P_\epsilon(t)}\le \epsilon \abs{t}$ for
  all $t\in\reals$, and therefore we have
  
  \begin{equation*}
    \abs{u(t)e^{-\abs{t}} - \frac{P_\epsilon(t)}{ (1+t^2)}} \le
    \epsilon \frac{\abs{t}}{1+t^2}\quad\implies \quad \abs{u(t)e^{-\abs{t}}
      -  \frac{P_\epsilon(t)}{1+t^2}} \le \epsilon\cdot C.
  \end{equation*}
  This shows that $(1+t^2)^{-1}P_\epsilon(t)$ approximates $\sqrt{\abs{t}}
  e^{-\abs{t}}$ in supremum norm.

  Finally,
  \begin{align*}
   \abs{\frac{u(t)}{\sqrt{\abs{t}}}e^{-\abs{t}} -
     \frac{P_\epsilon(t)}{\sqrt{\abs{t}}(1+t^2)} }_{L^2(\reals)} &\le \epsilon
   \left(\int_{-\infty}^\infty \frac{{\abs{t}}}{(1+t^2)^2} \,dt\right)^{1/2} \le
   \epsilon C,\\
\abs{\frac{\abs{u(t)e^{-\abs{t}}}}{\sqrt{\abs{t}}}- e^{-\abs{t}}}_{L^2(\reals)} &\le %\int_{-\epsilon}^\epsilon
  \abs{\frac{7}{4}\epsilon^{-3/2} t^2 - \frac{3}{4}\epsilon^{-7/4}t^4
    -1}_{L^2([-\epsilon,\epsilon])}\\%^2\;dt\\
 &\le \left( \int_{-\epsilon}^\epsilon
    \abs{\frac{7}{4}\frac{1}{\epsilon^{3/2}}t^2}^2 \,dt\right)^{1/2} + \left(\int_{-\epsilon}^\epsilon
    \abs{\frac{3}{4}\frac{1}{\epsilon^{7/2}}t^4}^2 \,dt\right)^{1/2}   +\left(\int_{-\epsilon}^{\epsilon} 1\,dt\right)^{1/2}
\le C\epsilon
 \end{align*}
with a constant $C$ independent of $\epsilon$. 
  \end{proof}
 
  \begin{definition}
    For $f\colon \reals\to\reals$ continuous and sufficiently decaying at
    $\infty$, define
    \begin{equation*}
V_f\colon L^2(M)\to
    L^2([0,\infty)\times M) \quad\text{ by }\quad V_fu(t):=
    \frac{1}{\sqrt{\abs{t}}}     f(t\abs{D+E}) 
      u,
    \end{equation*}
and correspondingly $V_{f,D}$ with $D+E$ replaced by $D$. 
  \end{definition}

  \begin{lemma}
    In the situation above, $\max\{\norm{V_D-V_{f,D}}^2,\norm{V-V_f}^2\} \le
    \int_0^{\infty} 
    \abs{\frac{f(t)}{\sqrt{t}}-e^{-t}}^2\,dt$% , and similarly
    % $\norm{V_D-V_{f,D}}\le \int_{-\infty}^0 \abs{\frac{f(t)}{\sqrt{\abs{t}}} -\exp(t)}^2\,dt$
    .
  \end{lemma}
  \begin{proof}
    By definition, $(V-V_f)u (t) = g_t(D+E) u$ with $g_t(\lambda) =
    \frac{\sqrt{2\abs{t\lambda}}e^{-t\abs{\lambda}}-f(t\abs{\lambda})}{\sqrt{{t}}}$. Then
    \begin{equation*}
      \int_0^{\infty} \abs{g_t(\lambda)}^2\;dt = \int_0^\infty
      \abs{\frac{1}{\sqrt{t}} 
    \left( \sqrt{2\abs{t\lambda}}         e^{-t\abs{\lambda}} - f(t\abs{\lambda})\right)}^2\,dt \stackrel{s=t\abs{\lambda}}{=}
  \int_0^{\infty}  \frac{1}{s}\abs{
    \sqrt{s} e^{-s} - f(s)}^2\,ds .
    \end{equation*}
  By Lemma \ref{lem:opnorm}, $\max\{\norm{V_D-V_{f,D}}^2,\norm{V-V_f}^2\} \le       \int_0^{\infty}
  \abs{g_t(\lambda)}^2\, dt   = \int_0^{\infty}  \abs{
     e^{-t} - \frac{f(t)}{\sqrt{t}}}^2\,dt$, as claimed.
 \end{proof}

Because of this, in light of Lemma \ref{lem:opnorm} we can find a polynomial
$f(t)$ in $(1+t^2)^{-1}$, vanishing at $\infty$ such that $V-V_f$ and
$V_D-V_{f,D}$ are arbitrarily small in norm. The same applies then of course
also to $(1-P_R) (V-V_f)$ and $(1-P_R)(V_D-V_{f,D})$.

\begin{lemma}\label{lem:decomp}
  For a fixed polynomial $f(t)$ in $(1+t^2)^{-1}$, $V_f-V_{f,D}$ is a linear
  combination of operators of the form
  \begin{equation}\label{eq:specform}
    \begin{split}
      % & u\mapsto \left(t\mapsto \frac{1}{\sqrt{\abs{t}}} \frac{1}{(1+t^2(D+E)^2)^l} (tE)^2
      % \frac{1}{(1+t^2D^2)^k} u \right)\\
      & u\mapsto \left(t\mapsto \frac{1}{\sqrt{\abs{t}}} \frac{t(D+E)}{(1+t^2(D+E)^2)^l} tE
      \frac{1}{(1+t^2D^2)^k}u \right)\\
     & u\mapsto \left( t\mapsto \frac{1}{\sqrt{\abs{t}}} \frac{1}{(1+t^2(D+E)^2)^l} tE
      \frac{tD}{(1+t^2D^2)^k} u\right)
    \end{split}
  \end{equation}
  with $k, l\ge 1$.
\end{lemma}
\begin{proof}
  First, expand $a^n-b^n = \sum_{k=1}^n a^{k-1}(a-b) b^{n-k}$. Apply this to
  $a=(1+t^2(D+E)^2)^{-1}$ and $b=(1+t^2D^2)^{-1}$. Then
  \begin{equation*}
    a-b = \frac{1}{1+t^2(D+E)^2}\left(1+t^2D^2- 1-t^2(D+E)^2 \right)  \frac{1}{1+t^2D^2},
  \end{equation*}
  and $1+t^2D^2-1-t^2(D+E)^2 = -t(D+E) tE -tE tD $.

  Recall that for $g(t)=\frac{1}{(1+t^2)^n}$ we have
 $\left((V_g-V_{g,D})u\right)(t) = 
  \frac{1}{\sqrt{\abs{t}}} \left(
    \left(\frac{1}{1+t^2(D+E)^2} \right)^n   - \left(\frac{1}{1+D^2}\right)^n  \right)u$. Application of the
  formulas just derived immediately gives the result.
\end{proof}

\begin{lemma}\label{lem:near_midterm}
  If $f(t)$ is a polynomial in $(1+t^2)^{-1}$ then
  \begin{equation*}
    \lim_{r\to 0} \norm{(1-P_r) (V_f-V_{f,D})} = 0.
  \end{equation*}
\end{lemma}
\begin{proof}
  Because of Lemma \ref{lem:decomp} it suffices to consider the operators
  given in \eqref{eq:specform}. Because of the factor $t$ in front of $tE$
  which is always present, they are of the form
  \begin{equation*}
    u\mapsto \left(t\mapsto \sqrt{t} A(t)u\right)
  \end{equation*}
  where $A(t)$ is a uniformly norm bounded family of operators: it is a composition
  of $E$ and of uniformly in $t$ operator norm bounded functions of $(D+E)$
  and of
  $D$. The statement follows immediately from the definition of the norm on
  $L^2([0,\infty),L^2(M))$. 
\end{proof}

\begin{lemma}\label{lem:prop_midterm}
  If $f(t)$ is a polynomial in $(1+t^2)^{-1}$ and  $0<r<R<\infty$ then
  $(P_r-P_R) (V_f-V_{f,D})$ is a norm limit of operators of finite
  propagation. 
\end{lemma}
\begin{proof}
  By Lemma \ref{lem:decomp}, it suffices {to compose}
  the operators in \eqref{eq:specform}
  with $(P_r-P_R)$ and prove the statement for these compositions.
  Now observe that these operators have the form $u\mapsto \left(t\mapsto
    A(t)u \right)$
  where $A(t)=
  \begin{cases}
    0; &  t>R,t<r\\
    \phi_t(D+E) E \psi_t(D); &r\le t\le R
  \end{cases}$.
Then $A(t) $ is a norm continuous function with values in operators which
  are norm limits of finite propagation operators: indeed, $\phi_t(\lambda),
  \psi_t(\lambda)$ tend to $0$ for $\lambda\to\pm\infty$ and depend continuously
  in supremum norm on $t$, so that $\phi_t(D+E)$, $\psi_t(D)$ are really
  limits of   finite propagation operators and depend norm continuously on
  $t$. Therefore we can ---up to an
  arbitrarily small error in norm--- replace the function $A(t)$ by a
  (say~piecewise constant) function $B(t)$ of operators with fixed finite
  propagation $S$. It follows that $u\mapsto \left(t\mapsto
    B(t)u\right)$ has finite  propagation at most $\max\{R,S\}$. 
  
\end{proof}

\begin{lemma}\label{lem:comp_midterm}
  If $f(t)$ is a polynomial in $(1+t^2)^{-1}$ and  $0<r<R<\infty$ then
  $(P_r-P_R) (V_f-V_{f,D})= B (i+rD)^{-1}$ for a bounded operator $B$
\end{lemma}
\begin{proof}
  Because of Lemma \ref{lem:decomp}, we have to show that
  \begin{equation*}
    u\mapsto \left(t\mapsto
      \begin{cases}
        0; & t>R \text{ or }t<r\\
        \frac{(t(D+E))^\eta}{(1+t^2(D+E)^2)^k} \sqrt{\abs{t}} E
        \frac{(tD)^{1-\eta} (i+rD)}{(1+t^2D^2)^l} u; & r\le t\le R
      \end{cases}\right)
  \end{equation*}
  with $k,l\ge 1$ and $\eta\in \{0,1\}$ is bounded. Using that $\eta=0$ or
  $\eta=1$ and $l\ge 1$ this follows from the fact that $\lambda\mapsto
  \frac{r\lambda}{1+t^2\lambda^2}$ and $\lambda\mapsto \frac{t\lambda
    r\lambda}{1+t^2\lambda^2}$ are uniformly (in $t\ge r$) bounded functions of
  $\lambda$ (note that we substitute $D$ into these functions to obtain one
  factor making up the operator we have to consider, the remaining factors
  being controlled by the previous considerations). 
\end{proof}

\begin{theorem}
  $V\colon L^2(M)\to L^2([0,\infty)\times M)$ covers the inclusion
  $\{0\}\times M\to
  [0,\infty)\times M$ in the $D^*$ sense.
\end{theorem}
\begin{proof}
  Choose a polynomial $f(t)$
  in $(1+t^2)^{-1}$ as above such that $(1-P_R)(V-V_f)$ and
  $(1-P_R)(V_{f,D}-V_D)$ have small norm for all $R>0$.
  We write
  \begin{equation*}
    V= P_R V + (1-P_R)(V - V_f) + (1-P_r) (V_f-V_{f,D}) +
    (P_r-P_R)(V_f-V_{f,D}) + (1-P_R) (V_{f,D}-V_D) + (1-P_R) V_D.
  \end{equation*}
  for a suitable choice of $R$ such that, using Lemma \ref{lem:awaysmall},
  $P_RV$ has small norm, and of $r$ such that, using Lemma
  \ref{lem:near_midterm}, $(1-P_r)(V_f-V_{f,D})$ has small norm.

 By Lemma \ref{lem:prop_midterm}, $(P_r-P_R)(V_f-V_{f,D})$ is a norm limit of
 finite propagation operators. Let $\phi$ be a compactly supported continuous
 function on 
 $M$ and $\psi$ a compactly supported continuous function on $
 [0,\infty)\times M$ with $\supp(\psi)\cap \{0\}\times\supp(\phi)=\emptyset$. Then 
 $(i+rD)^{-1}\phi\colon L^2(M)\to L^2(M)$ is compact and therefore, 
 by Lemma \ref{lem:comp_midterm} also $\psi(P_r-P_R)(V_f-V_{f,D})\phi$ is
 compact (pseudolocal condition). % {By Kasparov Lemma, see \cite[Lemma 3.8]{PS-Stolz}, we conclude
 % that the pseudocality condition, see \cite[Definition 1.7]{PS-Stolz}, holds for $(P_r-P_R)(V_f-V_{f,D})$.}
 {Finally, in \cite{PS-Stolz} we have shown, using unit propagation speed of
 the wave operator of $D$ on $M$ that $(1-P_R)V_D$ is a norm limit of finite
 propagation operators that satisfy, in addition, the pseudolocal condition}. Note that
 these derivations did not use invertibility of $D$ and therefore are valid in
 the present context.
{Summarizing, we have shown that $V$ is a norm limit of finite propagation operators $F_\epsilon$
with the additional property that   $\psi F_\epsilon  \phi$
is compact for 
any $\phi\in C_0 (M)$ and  $\psi\in C_0([0,\infty)\times M)$ with $\supp(\psi)\cap \{0\}\times\supp(\phi)=\emptyset$. This proves that $V$ covers
the inclusion in the $D^*$-sense. }
\end{proof}
The proof of Propositions \ref{prop:v-covers} is now complete.

\subsubsection{Proof of Propositions  \ref{prop:deformed-in-d}: the operator $\frac{\abs{D+C}+\partial_t}{D+C-\partial_t}$
  belongs to $D^* (\widetilde{M}\times\RR)^\Gamma$} 
\label{sec:belong_to_D}

\begin{notation}
 In all this subsection we shall denote by $M$ the total space of a 
Riemannian  $\Gamma$-Galois covering with compact base $M/\Gamma$.
We also consider a Riemannian manifold $N$.
(In the application we have in mind 
$N=\reals$.)

 We  consider a $\Gamma$-equivariant Dirac type operator $D$ on
$M$ acting on the sections of a Clifford module bundle; as before, we
will in the notation ignore this bundle. Similarly, we consider
a Dirac type operator $\partial$ on $N$ (and in the applications we have in
mind we shall in fact take 
 $\partial=i\partial_t$ on $N=\RR$). We recall that $D$ and $\partial$ are
 essentially 
 self-adjoint; we shall not distinguish notationally between $D$, $\partial$
 and their unique self-adjoint extensions.
\end{notation}

We wish to prove that the operator
$\frac{\abs{D+C}+\partial_t}{D+C-\partial_t}$ 
  belongs to $D^* (M\times\RR)^\Gamma$.
In the course of the argument, it turns out that it is useful to work not only
with $L^2$, but also with the Sobolev spaces $H^1, H^2$ on our complete %spin
manifold $M$, or $M\times N$. We have to understand mapping properties for
 perturbed Dirac operators and functions of those, acting
on these Sobolev spaces. Because we don't want to assume that our perturbation
 is a pseudodifferential operator, we can't use standard mapping properties
here; instead we will rely on abstract functional analysis of unbounded
operators on Hilbert spaces.

%We deal throughout with sections of the bundle of generalized spinors, on
%which our Dirac type operators act, without mentioning this
%explicitly.

\begin{definition}
For a manifold like $M$,  $H^1 (M)$ is defined as the domain of the (unique
self-adjoint extension of the) Dirac type operator $D$ endowed with the graph
norm $\left(\abs{Ds}^2+\abs{s}^2)\right)^{1/2}$.  Similarly, $H^2 (M)$ is the
domain of $D^2$, endowed  with the corresponding graph
norm.
 \end{definition}

\begin{proposition}
  Let $E$ be a self-adjoint bounded equivariant operator.
  Then $H^1 (M)$ coincides  with the domain of $D+E$ and its norm is
  equivalent to the graph norm 
  $\left(\abs{(D+E)s}^2+\abs{s}^2\right)^{1/2}$. 
  If $D+E$ is invertible, this norm is also
  equivalent to $\abs{(D+E)s}$.
In this case, $(D+E)^{-1}\colon L^2\to H^1$ is bounded, even an isometry.

Similarly, $H^2 (M)$ is equal to the domain of $D^2+1$ endowed with the graph norm; in
  particular $(D^2+1)\colon L^2\to H^2$ is  an isometry for the appropriate
  choice of norm on $H^2$.

Any suitably equivariant differential operator of order $1$ is a bounded operator $H^1\to L^2$.
In particular, on $M\times N$ both $D$ and $\partial$ are bounded operators
$H^1(M\times N)\to L^2(M\times N)$.

Finally,  $\abs{D}$ and $\abs{D+E}$ are bounded maps $H^1\to L^2$.
\end{proposition}

\begin{proof}
  Only the statements about $\abs{D},\abs{D+E}$ are not easy or 
  standard. For those, we can write $\abs{D+E} = \frac{\abs{D+E}}{D+E+i}
  (D+E+i)$,
  where the bounded function $\frac{\abs{D+E}}{D+E+i}$ of $D+E$ is a bounded
  operator on $L^2$ and $D+E+i$ is bounded from $H^1\to L^2$.
\end{proof}

\begin{proposition}\label{prop:partial_H1}
  Given the Riemannian product $M\times N$, a Dirac type operator
  $D$ on $M$ with bounded equivariant self-adjoint perturbation $E\colon
  L^2(M)\to L^2(M)$, 
  a compactly supported function $\phi$ on $N$ (acting by pointwise
  multiplication) and a compact 
  operator $K\colon L^2(M)\to L^2(M)$, the composition 
  \begin{equation*}
    H^1(M\times N)\xrightarrow{D+E} L^2(M\times N) =L^2(M)\tensor
    L^2(N)\xrightarrow{K\tensor \phi} L^2(M\times N)
  \end{equation*}
  is compact.
\end{proposition}
\begin{proof}
  Using the usual reduction techniques (write $\phi$ as a finite sum of
  functions with support in a coordinate neighborhood, use charts to plant
  these coordinate neighborhoods 
  into $T^l$) one reduces to the
  case where $N=T^l$. Then, the operator with general $\phi$ is the
  composition of the bounded operator on $L^2(M\times T^l)$ given by
  multiplication with $\phi$ with the special operator where $\phi=1$. It
  therefore suffices to show that the latter one is compact, and we set $\phi=1$.

  We now apply the strategy of the proof of the Rellich lemma. We have to
  understand a bit 
  better the domain $H^1(M\times T^l)$ which, as the domain of the (perturbed)
  Dirac 
  operator on $M\times T^l$ is a subspace of $L^2(M\times T^l)$ (with its own
  norm). We write $L^2(M\times T^l) = L^2(M)\tensor L^2(T^l)$. Using Fourier
  transform, we unitarily identify $L^2(T^l)$ with $l^2(\integers^l)$. Using
  the Browder-Garding spectral decomposition for the self-adjoint unbounded
  operator $D+E$, we write as a 
  direct integral $L^2(M)=\int d\mu(\lambda)\; H_\lambda$. By the definition of
  compact operators as norm limits of finite rank operators, we can replace
  $K\colon L^2(M)\to L^2(M)$
  up to an error of arbitrarily small norm by a finite rank operator
  $K_\Lambda$ such that
  $K_\Lambda$ maps $\int_{-\Lambda}^\Lambda H_\lambda $ to itself and is zero
  on the complement. By definition of the spectral decomposition, $D+E$ acts on
  the direct integral by multiplication with the spectral parameter
  $\lambda$. In particular, $\int_{-\Lambda}^\Lambda H_\lambda$ is entirely
  contained in the domain of $D+E$ (i.e.~in $H^1(M)$) and restricted to this
  subspace the norm of $D+E$ %\colon H^1\to L^2$ 
   is certainly
  bounded by the  norm of $D+E$ as a bounded operator from $H^1$ 
  to $L^2$. We denote this norm $C$.

  Thus the Hilbert space $H^1(M\times T^l)$  has the direct summand
  $\int_{-\Lambda}^\Lambda H_\lambda \tensor H^1(T^l)$. Here,
   after Fourier transform $L^2(T^l)\iso l^2(\integers^l)$, we identify
   $H^1(T^l)$ 
   with the domain of the operator
   \begin{equation*}
l^2(\integers^l)\to
  l^2(\integers^l);\;(\lambda_n)_{n\in\integers^k}\mapsto  
  (\abs{n}\lambda_n)_{n\in\integers},\qquad\text{with }\abs{n_1,\dots,n_l}=\abs{n_1}+\dots+\abs{n_l}
\end{equation*}
endowed with the
  graph norm. For $R>0$, split
  \begin{equation*}
  H^1(T^l):=V_R\oplus V_R^\perp\quad\text{ where}\qquad V_R=\{ (\lambda_n)_{n\in\integers^l}
  \mid \lambda_n=0\text{ if }\abs{n}>R\} \text{ is finite dimensional}.
\end{equation*}
 Note that  the
  inclusion $V_R^\perp\into l^2(\integers^l)$ has norm $<R^{-1}$.
 We now conclude the following:
 \begin{enumerate}
 \item The operator $(K\tensor \id_{L^2(T^l)})\circ (D+E)$$\colon H^1(M\times T^k)\to
   L^2(M)\tensor L^2(T^k)$ is norm close to $(K_\Lambda\tensor
   \id_{L^2(T^l)})\circ (D+E)$.
 \item restricted to the direct summand $\int_{-\Lambda}^\Lambda H_\lambda
   \tensor V_R$ of $H^1(M\times T^l)$, $(K_\Lambda\tensor \id) \circ (D+E)$ has
   finite rank with image $\im(K_\Lambda)\tensor V_R$.
 \item restricted to the direct summand $\int_{-\Lambda}^\Lambda H_\lambda
   \tensor V_R^\perp$, the operator $D+E\colon H^1(M\times T^k)\to L^2(M\times
   T^k)$ has norm 
   $\le C \cdot R^{-1}$, where %$\Lambda$ bounds the norm of $D+E$
   %restricted to $\int_{-\Lambda}^\Lambda H_\lambda$ and 
   $R^{-1}$ 
   comes from
   the ratio of the $H^1$-norm and the $L^2$-norm on $V_R^\perp$.

   Finally, on the orthogonal complement of $\int_{-\Lambda}^{\Lambda}
   H_\lambda\tensor H^1(T^k)$ in $H^1(M\times T^k)$, by the choice of
   $K_\Lambda$, the operator $(K_\Lambda\tensor \id)\circ (D+E)$ vanishes.
 \item It follows that $(K_\Lambda\tensor\id)\circ (D+E)$ is up to an error of
   norm $C/R$ a 
   finite rank operator.
 \item All together, $(K\tensor \id) \circ (D+E)\colon H^1(M\times T^l)\to
   L^2(M\times T^l)$ is a norm limit of finite rank operators, i.e.~is compact.
 \end{enumerate}

\end{proof}

\begin{proposition}\label{prop:D}
 If $\phi$ stands for the multiplication operator with the compactly
  supported $C^1$-function $\phi$ and $P$ for any first order equivariant
  differential operator, the commutator $[\phi,P]$ is compact as operator
  from $H^1(M)\to L^2(M)$.

The same applies if $P$ is replaced by $E$ or $P+E$ for any equivariant
self-adjoint bounded operator $E$ which is a norm limit of operators with
finite propagation.
\end{proposition}
\begin{proof}
  The commutator $[\phi,P]$ is a multiplication operator
  with a derivative of $\phi$, a compactly supported function. The latter ones
  are compact as maps from $H^1$ to $L^2$.

  For the perturbation, we have to consider also the commutator $\phi E-E\phi$. By
  finite propagation, up to a norm-small error we can write $ \phi E = \phi E
  \psi$ with a compactly supported $\psi$. Then we only need to use that
  $H^1\to L^2\xrightarrow{\phi} L^2$ and $H^1\to L^2\xrightarrow{\psi} L^2$ are
  compact by the Rellich lemma and that $E\colon L^2\to L^2$ is bounded.
\end{proof}

\begin{proposition}\label{prop:DE}
  Let $D$ be an equivariant Dirac type operator on $M$ and $E\colon L^2(M)\to
  L^2(M)$ 
  an $L^2$-bounded self-adjoint perturbation which is a norm limit
  of equivariant finite propagation operators such that $D+E$ becomes 
  invertible. Assume that
  $\phi$ is a compactly supported $C^1$-function on
  $M\times N$. {Then $\abs{D+E}$, as map from
  $H^1(M\times N)\to L^2(M\times N)$, is a limit of equivariant bounded
  operators $F_\epsilon\colon
  H^1(M\times N)\to L^2(M\times N)$ that are of finite propagation and such that $[F_\epsilon,\phi]
  \colon
  H^1(M\times N)\to L^2(M\times N)$
   is compact .}
\end{proposition}
\begin{proof}
  {We can reduce to $\phi=\alpha(x)\beta(y)$ with $\alpha\colon M\to\complexs$,
  $\beta\colon N\to \complexs$ compactly supported $C^1$-functions. 
  We write $\abs{D+E}\colon H^1(M\times N)\to L^2(M\times N)$
  as 
  $$H^1(M\times N)\xrightarrow{D+E} L^2(M\times N)=L^2(M)\tensor
  L^2(N)\xrightarrow{(\abs{D+E}/(D+E))\tensor \id} L^2(M)\tensor L^2(N)$$
  {Using  Lemma \ref{lem:ancon}}, write 
  $\abs{D+E}/(D+E)$ on $M$ as the limit of equivariant finite propagation
  pseudolocal operators $\kappa_\epsilon$ and 
  consider $F_\epsilon$ given by the composition
   $$F_\epsilon\colon H^1(M\times N)\xrightarrow{D+E} L^2(M\times N)=L^2(M)\tensor
  L^2(N)\xrightarrow{\kappa_\epsilon\tensor \id} L^2(M)\tensor L^2(N)$$
  Observe that  $F_\epsilon$ is of finite propagation as a
  bounded operator $H^1(M\times N)\to L^2(M\times N)$. We now show that it is also pseudolocal
  as a bounded  operator $H^1(M\times N)\to L^2(M\times N)$. We have
  \begin{equation*}
    [\phi,F_\epsilon] = \left([\alpha, \kappa_\epsilon]\right) (D+E) \,\tensor \,\beta \;+ \; \kappa_\epsilon
    \left(
    [\alpha, (D+E)]\right)\,\tensor \,\beta.
  \end{equation*}
Use here that $\beta$
 commutes with all the other operators, which allows to split off the tensor
 factor $\beta$ in $\phi=\beta\alpha$ throughout.
 The second summand is compact by Proposition~\ref{prop:D}. The first summand is compact by
 Proposition~\ref{prop:partial_H1}.}
 \end{proof}

\begin{proposition}\label{prop:finpro}
  Given a Dirac type operator $D$ as above, the operator $(1+D^2)^{-1}\colon
  L^2\to H^2$ is {a norm limit of equivariant
  finite propagation operators $G_\epsilon\colon L^2\to H^2$ with the property that
  $[\phi, G_\epsilon]\colon L^2\to H^2$ is compact for any compactly supported
  smooth function on $M$.}
\end{proposition}

\begin{proof}
 {We follow the idea of the proof of Proposition 4.19 in \cite{PS-Stolz}. Consider $f(x):=(1+x^2)^{-1}$
 and arrange that $f(x)= g_\epsilon (x)+h_\epsilon (x)$ where $g_\epsilon (x)$ has
  compactly supported Fourier transform and the Fourier transform of
  $h_\epsilon(x)$ together with its second derivative {have small $L^1$-norm; use here that $\hat f(\xi)$,
  albeit non-smooth at $0$, is away from $0$ smooth such that all derivatives
  are rapidly decreasing. Set $G_\epsilon:= g_\epsilon (D)$.
 Then $G_\epsilon$ has, by unit propagation of the wave operator, finite
  propagation, and 
    \begin{multline*}
      \norm{(1+D^2)^{-1} - G_\epsilon}_{L^2\to
        H^2}=\norm{h_\epsilon(D)}_{L^2\to H^2} = \norm{(1+D^2) h_\epsilon
        (D)}_{L^2\to L^2} \le \abs x\mapsto (1+x^2) h_\epsilon
          (x)_\infty\\ \le \abs{\hat h}_1+ \abs{\hat h''}_1 << 1.
    \end{multline*}
  The fact that $G_\epsilon$ is pseudolocal as a map $L^2\to H^2$ is proved by employing
  a doubling trick and 
  (standard) microlocal techniques, as in 
  \cite{PS-Stolz}.}}
\end{proof}
A purely functional analytic argument, which applies to Lipschitz manifolds
without a pseudodifferential calculus is given in \cite{zenobi}.

\begin{proposition}\label{prop:finproful}
  Let $D$ be a Dirac type operator on $M$ with equivariant self-adjoint
  bounded perturbation $E\colon L^2(M)\to
  L^2(M)$ as above such that $D+E$ is invertible and such that $E$ is the norm
  limit of equivariant finite propagation operators. {Then
  $(D+E)^{-1}\colon L^2\to H^1$ is a norm limit of equivariant finite propagation
  operators $A_\epsilon\colon L^2\to H^1$ with the property that for
  any compactly supported smooth function $\phi$ the operator $[\phi,A_\epsilon]\colon L^2\to
  H^1$ is compact.}
\end{proposition}
\begin{proof}
%  To prove that $[(D+E)^{-1},\phi]\colon L^2\to H^1$ is compact, we compose
%  with the isometry $(D+E)\colon 
%  H^1\to L^2$ and show that the composition $L^2\to L^2$ is compact.
%  This composition is $(D+E) (\phi (D+E)^{-1} - (D+E)^{-1}\phi)=
%  ((D+E)\phi -\phi (D+E)) (D+E)^{-1}$. As $(D+E)^{-1}\colon L^2\to H^1$ is
%  bounded and $[D+E,\phi]\colon H^1\to L^2$ is compact by Proposition
%  \ref{prop:D}, the compactness assertion follows.
 {Write
  \begin{equation*}
(D+E)^{-1} = 
  (D+i)(1+D^2)^{-1}\left((D+E-i)(D+E)^{-1}-E(D+E)^{-1}\right).
\end{equation*}
Here, $(D+E-i)(D+E)^{-1}\colon L^2\to L^2$ is, by Lemma \ref{lem:ancon}, an element in $D^*$, thus 
a norm limit of equivariant finite propagation pseudolocal operators
$h_\epsilon$. Similarly $(D+E)^{-1}\colon L^2\to L^2$ and therefore
$E(D+E)^{-1}\colon L^2\to L^2$ are elements in $C^*$,
and so the latter is a norm limit of 
equivariant finite propagation locally compact operators
$\lambda_\epsilon$. Next,
$D+i\colon H^2\to H^1$ is bounded with propagation $0$ and has the property that
$[\phi,D+i]\colon H^2\to H^1$ is compact by the Rellich lemma. Finally, by
Proposition  \ref{prop:finpro} above,
$(1+D^2)^{-1}\colon L^2\to H^2$ is the norm limit of equivariant finite
propagation pseudolocal
operators $G_\epsilon \colon L^2\to H^2$. Using the derivation property of $[\phi,\cdot]$ 
and approximating $(D+E)^{-1} $ by $A_\epsilon:=(D+i) G_\epsilon \left( h_\epsilon - \lambda_\epsilon \right)$
we easily conclude the proof.}
\end{proof}

\begin{corollary}
  If $E\colon L^2(M)\to L^2(M)$ is a bounded self-adjoint equivariant
  operator and a
  norm limit of equivariant finite propagation
  operators and $D$ is a Dirac type operator on $M$ as above such that $D+E$
  is invertible, 
  then the operator $\frac{\abs{D+E}+\pa_t}{D+E-\pa_t}\colon L^2\to L^2$
  belongs to $D^* (M\times \RR)^\Gamma$.
\end{corollary}
\begin{proof}
{We must prove that $\frac{\abs{D+E}+\pa_t}{D+E-\pa_t}$ is the norm limit
of a sequence of equivariant bounded operators $H_\epsilon:L^2 \to L^2$ that
are of finite propagation and pseudolocal. To check pseudolocality, by a
density argument it suffices to consider the commutator with compactly
supported smooth functions.}
We choose $N=\RR$ in the previous propositions. We observe that
  $(D+E-\pa_t)$ is a summand of the perturbed Dirac type operator $\left(
    \begin{smallmatrix}
      0 & D+E-\pa_t \\ D+E+\pa_t & 0
    \end{smallmatrix}\right)$ 
on $M\times \reals$;
  thus $(D+E-\pa_t)^{-1}\colon L^2\to H^1$ is a {norm limit of
    equivariant operators  
  $A_\epsilon\colon L^2\to H^1$ that are of
  finite propagation 
  operators and
  commute up to compact operators with multiplication by compactly supported
smooth  functions. Here we have used Proposition \ref{prop:finproful} with
$M\times \RR$   instead of $M$.
By Proposition \ref{prop:DE}, 
$\abs{D+E}$, as map from
  $H^1(M\times \reals)\to L^2(M\times \reals)$, is a limit of bounded
  equivariant  operators $F_\epsilon\colon 
  H^1(M\times \reals)\to L^2(M\times \reals)$ that are of finite propagation and such that $[F_\epsilon,\phi]
  \colon
  H^1(M\times \reals)\to L^2(M\times \reals)$
   is compact. Finally, $\pa_t: H^1\to L^2$ has propagation zero and $[\phi,\pa_t]$  is clearly
   compact as a map $H^1\to L^2$ (Rellich Lemma). Summarizing, by writing
   $$\frac{\abs{D+E}+\pa_t}{D+E-\pa_t} = \frac{\abs{D+E}}{D+E-\pa_t} + \frac{\pa_t}{D+E-\pa_t}$$
   we see that $\frac{\abs{D+E}+\pa_t}{D+E-\pa_t}$ is a sum of two elements in  $D^* (M\times \RR)^\Gamma$
   and it is therefore 
   in  $D^* (M\times \RR)^\Gamma$, as required.}
\end{proof}

\bibliography{surgery}
\bibliographystyle{plain}

\end{document}